\documentclass[a4paper,twoside,12pt]{amsart}

\usepackage{amssymb}
\usepackage{amsmath}
\usepackage{epsfig}
\usepackage{graphicx}
\usepackage{hyperref}
\usepackage[nooneline]{subfigure}
\usepackage{color}

\usepackage[latin1]{inputenc} 
\usepackage[T1]{fontenc}

\newtheorem{theorem}{Theorem}[section]
\newtheorem{fthm}[theorem]{Factorization Theorem}
\newtheorem{lemma}[theorem]{Lemma}
\newtheorem{proposition}[theorem]{Proposition}
\newtheorem{cor}[theorem]{Corollary}
\newtheorem{qst}[theorem]{Question}

\theoremstyle{definition}
\newtheorem{definition}[theorem]{Definition}
\newtheorem{example}[theorem]{Example}
\newtheorem{notation}[theorem]{Notation}
\newtheorem{caso}{Case}
\newtheorem{caso0}{Case}
\newtheorem{caso1}{Case}
\newtheorem{vertex1}{First Vertex}
\newtheorem{vertex2}{Second Vertex}
\newtheorem{subcaso}{Sub-Case}
\newtheorem{paso}{Step}

\DeclareMathOperator{\supp}{Supp}
\DeclareMathOperator{\In}{In}
\DeclareMathOperator{\depth}{depth}
\newcommand\balpha{\boldsymbol{\alpha}}
\newcommand\bdx{{\mathbf x}}
\newcommand\bdy{{\mathbf y}}
\newcommand\bn{{\mathbb N}}

\newcommand\bq{{\mathbb Q}}

\newcommand\bk{{\mathbb K}}

\newcommand\bdn{{\mathbf n}}

\newcommand\bdm{{\mathbf m}}

\numberwithin{equation}{section}

\newcommand\res{\operatorname{Res}}

\theoremstyle{remark}
\newtheorem{remark}[theorem]{Remark}

\numberwithin{equation}{section}

\newcommand\enet[1]{\renewcommand\theenumi{#1}
\renewcommand\labelenumi{\theenumi}}

\date{\today}

\title{Quasi-ordinary singularities and Newton trees}

\author[E. Artal]{E. Artal Bartolo}
\address{Departamento de Matem\'aticas-IUMA, Universidad de Zaragoza,
c/ Pedro Cerbuna 12, 50009 Zaragoza, SPAIN}
\email{artal@unizar.es}
\thanks{}

\author{Pi. Cassou-Nogu\`es}
\address{Institut de Math\'ematiques de Bordeaux, Universit\'e Bordeaux I, 
350, Cours de la Lib\'eration, 33405, Talence Cedex 05, FRANCE}
\email{cassou@math.u-bordeaux1.fr}
\thanks{}
\author{I. Luengo}
\address{Dept.\ of Algebra, 
Facultad de Ciencias Matem\'aticas, Universidad Complutense, 28040, Madrid,  SPAIN}
\curraddr{}
\email{iluengo@mat.ucm.es}
\thanks{}

\author[A. Melle]{A. Melle Hern\'andez}
\address{ICMAT (CSIC-UAM-UC3M-UCM) \\
Dept.\ of Algebra, 
Facultad de Ciencias Matem\'aticas, Universidad Complutense, 28040, Madrid,  SPAIN}
\curraddr{}
\email{amelle@mat.ucm.es}
\thanks{}

\thanks{
First and second authors are partially supported by
MTM2010-21740-C02-02; the last three authors are partially
supported by the grant 
MTM2010-21740-C02-01.}



\begin{document}
\begin{abstract}
In this paper we study some properties of the class of $\nu$-quasi-ordinary hypersurface singularities. 
They are defined by a very mild condition on its (projected) Newton polygon. We associate with them a Newton tree
and characterize quasi-ordinary hypersurface singularities among  
$\nu$-quasi-ordinary hypersurface singularities in terms of 
their Newton tree.
A formula to compute the discriminant of a quasi-ordinary Weierstrass polynomial in terms 
of the decorations of its Newton tree is given.  This allows to compute the discriminant avoiding the use
of determinants and even for  non Weierstrass prepared polynomials.
This is important for applications like algorithmic resolutions. 
We compare the Newton tree of a quasi-ordinary singularity and those of its curve 
transversal sections. We show that the Newton trees of the transversal sections  
do not give the tree of the quasi-ordinary singularity in general. 
It does if we know that the Newton tree of the quasi-ordinary singularity has only one arrow. 
\end{abstract}

\maketitle

\section*{Introduction}\label{sec-intro}
Let $\bk$ be an algebraically closed field of characteristic zero. Let $f\in \bk[[\bdx]][z]$ be a polynomial
with coefficients  in the formal power series ring $\bk[[\bdx]]:=\bk[[x_1,\cdots, x_d]]$.

If $d=1$, its zero locus $(\{f=0\},0)\subset (\bk^2,0)$ defines  a germ  of plane curve singularity. One of the  
initial ideas to deal with  germs of plane curve singularities is  the use of the Newton algorithm. 
Inside the algorithm, the first step is governed by the Newton polygon of $f$
and, for each compact face of the Newton polygon, one proceeds 
by doing the Newton process, 
and so on. This procedure is codified in a tree called \emph{Newton tree}. Many properties can be codified in the tree, for instance,  
using bi-colored Newton trees one can compute the intersection multiplicity of two plane curves. 
This allows to compute from the Newton tree the Milnor number of a germ since it can be expressed 
in terms of intersection multiplicities. More recently  for a finite codimension ideal $I$ in $\bk[[x_1,x_2]]$
its  multiplicity, its {\L}ojasiewicz exponent, 
and its Hilbert-Samuel multiplicity can be computed from its Newton tree, see~\cite{cnv:12}.   

For $d>1$ the well-known class  of hypersurface singularities  which generalizes the case of curves is the class of 
\emph{quasi-ordinary hypersurface singularities}. 
In that class, one of the key results is a factorization theorem given by Jung-Abhyankar, see \cite{ab:55}. 
One feature of quasi-ordinary singularities is that in some coordinates their Newton polyhedron is a polygonal path, i.e. 
all its compact faces have dimension one. Then we can apply the Newton process as in the case of curves. 
Moreover, after Newton maps, the condition to be quasi-ordinary is preserved. In \cite{aclm:09}, 
we have studied a generalization of quasi-ordinary singularities, called $\nu$-quasi-ordinary 
by H.~Hironaka \cite[Definition 6.1]{hr:74}. 
Roughly speaking  for a $\nu$-quasi-ordinary hypersurface defined by a polynomial $f\in \bk[[\bdx]][z]$  we only require the upper part of the Newton polyhedron of $f$
to be  a polygonal path. We apply the Newton process associated with these faces of dimension~$1$ and we iterate the process 
whenever the upper part of the corresponding Newton polyhedron is a polygonal path. 
As in the case of curves we encode the 
Newton process in a tree, but the tree bears leaves (arrows) and fruits (black boxes). Using this tree, we describe the 
 condition for two 
$\nu$-quasi-ordinary  series $f$ and $g\in \bk[[\bdx]][z]$ to have $z$-resultant which is a monomial times a unit.
In particular, a quasi-ordinary singularity has a Newton tree with only arrows.

The first main result of this article is that a $\nu$-quasi-ordinary  polynomial $f$  is a quasi-ordinary 
if and only if and only if there exists a suitable system of coordinates for $f$ such that
its Newton tree  has only arrow-heads decorated with $(0)$ and $(1)$ and has no black boxes, see Theorem~\ref{main2}. 
This result can also be stated as follows: 
\begin{quote}\em
Let $f$ be a $\nu$-quasi-ordinary polynomial. Then, $f$  is a quasi-ordinary
polynomial if and only if $f$ is reduced, $\nu$-quasi-ordinary
and the succesive transforms by Newton maps are also $\nu$-quasi-ordinary. 
\end{quote}
 This gives a characterization of
quasi-ordinary polynomial which does not depend on the definition of Newton trees, that is, quasi-ordinary
is equivalent to reduced and stably (by Newton maps) $\nu$-quasi-ordinary.
Theorem~\ref{main2} is proved studying how the derivative ${\partial f}/{\partial z}$ separates from~$f$ on the Newton 
tree of the product.  
In the case of curves this is an  algebraic elementary way to recover L\^e-Michel-Weber theorem~\cite{lmw:89}. 
We also give the formula to compute the discriminant of a quasi-ordinary Weierstrass polynomial $f$ in terms 
of the decorations of the Newton tree of $f$, see Proposition \ref{discrimant-tree}. This formula has two
important consequences. The first is that this formula allows to compute the discriminant avoiding the use
of determinants. The second one, if one starts with a non Weierstrass polynomial, the computation
of the Weierstrass decomposition is not needed. 
This is important for applications, for instance  algorithmic resolutions, see \cite{b:09,brj:03,bs:00}. 
In \cite{b:09} the author gives an effective algorithm for desingularization of surfaces using Jung's method. 
His Algorithms~2 and~3 use as  input 
a monic polynomial. If the polynomial is not monic, it  is a difficult task to make effective the Weierstrass preparation theorem. 
Using our method and proceeding as in   \cite{alr:88}
one can  avoid this problem and  adapt Algorithms 2 and 3 without the monic condition.  

The second main result is contained in \S\ref{sec-trans-sect} 
where we compare the Newton tree of a quasi-ordinary singularity and those of its curve 
transversal sections. Examples~\ref{ex-duple} and~\ref{ex-duple1} show that, in general, it is not possible to recover 
the Newton tree of the quasi-ordinary singularity from the Newton trees of the curve transversal sections. 
It does if we know that the Newton tree of the quasi-ordinary singularity has only one arrow, see Theorem~\ref{thm-recover}. 
However the global decorations which appear 
on the Newton tree of the quasi-ordinary singularities are those which appear on the Newton tree 
of the transversal sections. This fact plays a crucial role in the proof of 
the monodromy conjecture for quasi-ordinary singularities \cite{aclm:05}.
Moreover in \cite{GPGV:11} a description
of the motivic Milnor fibre of an irreducible quasi-ordinary polynomial 
is given proving that it is a topological invariant.  In~\cite{gv:12} 
the notion of linear Newton tree for an irreducible quasi-ordinary polynomial is introduced and some 
of its properties are discussed, as normalization or semigroup.

In \S\ref{sec-nuqonewton} we recall the definitions and some properties of  
$\nu$-quasi-ordinary polynomials and Newton maps. The construction of the Newton trees introduced
in \cite{aclm:09} is recalled and explained in \S\ref{sec-trees}; the notion of comparable polynomials
and the computation of resultants are also recalled in this section. In order to discuss
the unicity of Newton trees we work with the notion of P-good coordinates in \S\ref{P-good-coor-sec}, introduced
by P.~Gonz{\'a}lez-P{\'e}rez in~\cite{go:01}. The last two sections are devoted to state and prove the main
results of the paper.

\section{On $\nu$-quasi-ordinary polynomials and Newton maps}\label{sec-nuqonewton}
\medskip

\subsection{Basic facts on Newton polyhedra} 
\mbox{}

We shall follow the terminology of  \cite{aclm:05,aclm:09}. 
The $d$-tuples will be denoted in bold letters, e.g. $\bdx:=(x_1,\dots,x_d)$, and we will use the following
notations:
\begin{itemize}
\item $\bdx^{\balpha}:=x_1^{\alpha_1}\cdot\ldots\cdot x_d^{\alpha_d}$; 
\item $\mathbf{p\cdot q}:=p_1 q_1+\dots+p_d q_d$;
\item $\mathbf{p*q}:=(p_1 q_1,\dots,p_d q_d)$;
\item $\mathbf{\dfrac{p}{q}}:=\left(\dfrac{p_1}{q_1},\dots,\dfrac{p_d}{q_d}\right)$.
\end{itemize}

For Newton theory, we shall also follow the terminology of \cite{agv,kus,W:04}; note that for the terms
\emph{polyhedron} and \emph{diagram} we follow the convention in \cite{aclm:05,aclm:09,W:04};
the terms are exchanged in \cite{agv}.
Let  $\mathbb{N}\subset \mathbb{R} _+$ be the sets of non-negative integers and 
non-negative real numbers  respectively. Let $E\subset \mathbb{N}^{d+1}$ be a set of points and  $d\geq 1$.

\begin{itemize}
\smallbreak\item  The \emph{Newton diagram} $\mathcal{N}_+(E) $ is defined by the  convex hull in $(\mathbb{R} _+)^{d+1}$ of the set 
 $E+(\mathbb{R} _+)^{d+1}$.
\smallbreak\item The \emph{Newton polyhedron}  $\mathcal{N}(E)$
of $E$ is  defined by the union of all 
compact faces of the Newton diagram of $E$.
\smallbreak\item The smallest set $E_0$ such that
$\mathcal{N}_+(E_0) =\mathcal{N}_+(E) $ is called the \emph{set of vertices} of $E$.
\smallbreak\item A diagram is called \emph{polygonal} if the maximal dimension of its compact faces is one.
\end{itemize}
  
\smallbreak

Let $\bk$ be an algebraically closed field of characteristic zero. We write 
$f(\bdy)=f(y_1,\ldots,y_{d+1})$ for a formal power series   of several  variables   $\bk[[\bdy]]$. 
If the Taylor expansion of $f(\bdy)$
is $\sum_{\balpha\in\bn^{d+1}} c_{\balpha} \bdy^{\balpha }$, 
then  the \emph{support} $\supp (f)$ is defined to be $
\{\balpha\in \mathbb{N}^{d+1} \vert  c_{\balpha}\neq 0\}$.

\begin{itemize}
\smallbreak\item The \emph{Newton polyhedron} $\mathcal{N}(f)$ of $f$   
(resp. the \emph{Newton diagram} $\mathcal{N}_+(f) $) of $f$ is defined
 by the Newton polyhedron (resp. the Newton  diagram)
of the set $\supp (f)$.

\smallbreak\item If $\gamma $ is a compact face of $\mathcal{N}_+(f)$ then  the 
(weighted-homogeneous polynomial)  $f_{\gamma }(\bdy):=
\sum _{(\balpha )\in \gamma }
c_{\balpha }\bdy^{\alpha } \in \bk[\bdy]$ is  called 
the \emph{polynomial associated with $\gamma $}.
\smallbreak\item If the Newton diagram  $\mathcal{N}_+(f) $ is polygonal we also say that 
the polyhedron $\mathcal{N}(f)$
is a \emph{monotone polygonal path}.
\end{itemize}

If $f=f_1\cdot\ldots\cdot f_s$ in $\bk[[\bdy]]$ then 
$$
\mathcal{N}_+(f)=\mathcal{N}_+(f_1)+\ldots+\mathcal{N}_+(f_s),
$$
where the sum is  the Minkowski sum of Newton diagrams. This implies the following result which is well-known for 
experts, e.g. see \cite[Lemma~12]{gg:05}.

\begin{lemma}\label{polygonal}
If $f\in \bk[[\bdy]]$ has a Newton polyhedron which is a monotone polygonal path, then any irreducible factor of $f$ 
which is not associated with $y_i$,
$i=1\ldots, {d+1}$, has a Newton polyhedron which is a monotone polygonal path.
\end{lemma}

Let $\gamma $ be a 1-dimensional compact face of $\mathcal{N}(f)$. There exist two integral points $A, A_1$ in  
$\mathbb{N}^{d+1}$ such that $\gamma$ is the edge $[A,A_1]$.  Let $c$ denote the greatest common divisor  
of all coordinates of the vector $A_1-A$ and let
$\mathbf{u}:=\frac{1}{c}(A_1-A)\in \mathbb{Z}^{d+1}$.
The number of points with integer coordinates  on~$\gamma$ is $c+1$
and any of them is of the form $A_j=A+j\mathbf{u}$, 
with $0\leq j\leq c$. Then
\begin{equation}\label{factor-cara}
f_{\gamma }(\bdy)=\sum_{a\in \gamma} c_a \bdy^{a }=\bdy^A(\sum_{j=0}^{c} c_{A_j}\bdy^{j u})=
\bdy^A p(f,\gamma)(\bdy^u)
\end{equation}
where $p(f,\gamma)(t)$ is the polynomial  $p(f,\gamma)(t):=\sum_{j=0}^c c_{A_j}t^{j}$ 
of degree $c$; since $A$ and $A_1$ are vertices of $\gamma$ we have  $p(f,\gamma)(0)\ne 0$. 
Let us factor $p(f,\gamma)(t)$ as
\begin{equation}\label{factor-cara-1}
p(f,\gamma)(t)=a_\gamma\prod (t-\mu_i)^{m_i},\ 
\sum m_i=c,\
\mu_i\in \bk^*,\ 
\mu_i\ne \mu_j, \text{ if }i\neq j,
\text{ and }
a_\gamma\in \bk^*.
\end{equation}
Let us summarize these facts.
\begin{proposition}\label{factor-cara-2}
Let $\gamma =[A,A_1]$ be a 1-dimensional compact face of $\mathcal{N}(f)$. With the notations
of \eqref{factor-cara} and \eqref{factor-cara-1} the
 polynomial  $f_{\gamma }(\bdy)$ decomposes as $a_\gamma \bdy^A \prod(\bdy^u-\mu_i)^{m_i}$ in $\bk((\bdy))$. 
\end{proposition}

\begin{theorem}\cite[Theorem~3]{gg:05} If $f$ in $\bk[[\bdy]]$ is irreducible and has a  Newton polyhedron 
$\mathcal{N}(f)$ which is a monotone polygonal path,
then the diagram $\mathcal{N}_+(f)$ has only one compact edge $\gamma$ and the polynomial $p(f, \gamma)$ has only one
root $\mu_f$ in $\bk^*$, i.e. $p(f,\gamma)=a_\gamma(t-\mu_f)^{c},$ and  $a_\gamma\in \bk^*$.
\end{theorem}

\subsection{On $\nu$-quasi-ordinary polynomials }
\mbox{}

\begin{notation}\label{inicial}
Let 
$f(\bdx,z):=\sum c_{\balpha ,\beta} \bdx^{\balpha }z^{\beta } \in \bk[[\bdx]][z]$, 
$\balpha\in\bn^d$, $\beta\in\bn$, be a 
$z$-polynomial with coefficients in the formal
power series ring $\bk[[\bdx]]$, 
$\bdx:=(x_1,\cdots, x_d)$. Since the ring $\bk[[\bdx]][z]$ is factorial,
we may assume that $f(\bdx,z)=x_1^{n_1}\cdot\ldots\cdot x_d^{n_d} g(\bdx,z)$ where $g(\bdx,z)$ is regular 
of order say $n\geq 0$, that is 
$g(\mathbf{0},z)=a_0 z^n+a_1 z^{n+1}+\dots $, $a_0\in \bk^*$. 
Applying  Weierstrass preparation theorem  to $g(\bdx,z)$ there exists a unit $u(\bdx,z) \in \bk[[\bdx,z]]$
and a Weierstrass polynomial $h(\bdx,z)=z^n+a_1(\bdx)z^{n-1}+\ldots+ a_{n-1}(\bdx)z+a_0(\bdx)$
with $a_i(\textbf{0})=0$ such that $g(\bdx,z)=h(\bdx,z)u(\bdx,z)$ and $u(\textbf{0},0)=a_0\in \bk^*$.
e.g. see \cite[Chapter I, p. 11]{GLS}. Note that $\mathcal{N}(g)=\mathcal{N}(h)$ and we do not need to
know explicitely~$h$ for the constructions in this paper.
\end{notation}

\begin{remark}\label{1-vertex}
Let $f(\bdx,z):=\sum c_{\balpha ,\beta} \bdx^{\balpha }z^{\beta }=
x_1^{n_1}\cdot\ldots\cdot x_d^{n_d} g(\bdx,z)  \in \bk[[\bdx]][z]$ 
be a polynomial and assume we are as in Notation \ref{inicial}. 
Its Newton polyhedron $\mathcal{N} (f)$ consists only of one vertex
$(\bdn,n):=(n_1,\ldots,n_d,n)$ if and only if $\mathcal{N}_+ (f)=(\bdn,n)+(\mathbb{R} _+)^{d+1}$.
\end{remark}

\begin{definition}\label{eliminated}
A compact face ${\gamma }$ of the Newton polyhedron $\mathcal{N}(f)$  \emph{can be eliminated} if the polynomial 
$f_{\gamma }$ associated with $\gamma $ can be written as 
$f_{\gamma }=\bdx^{\mathbf{m}} (z-h(\bdx))^n \in \bk[\bdx,z]$, with $h\in \bk[\bdx], n\geq 1$. 
In such a case and by applying  the change of variables map
$\sigma:\bk[[\bdx]][z]\to \bk[[\bdx]][z_1]$
defined
by $z\mapsto z_1=z+h(\bdx)$ to $f$ then the face $\gamma$ is eliminated in 
the new Newton diagram $\mathcal{N}_+ (f\circ \sigma)$.

\end{definition}

\begin{definition}\label{inicial2}
Let $f(\bdx,z)=\sum c_{\balpha ,\beta} \bdx^{\balpha }z^{\beta }=
x_1^{n_1}\cdot\ldots\cdot x_d^{n_d} g(\bdx,z)  \in \bk[[\bdx]][z]$ 
be a polynomial and assume we are as in Notation \ref{inicial} and denote 
$A:=(\bdn,n)\in \mathbb{Q}^{d+1}$.
Let 
$$
\pi_A:\mathcal{N}_+(f)\setminus A\to \mathbb{Q}^d\equiv\mathbb{Q}^{d+1}\cap\{z=0\}
$$ 
be the projection into $\mathbb{Q}^d
$ 
with centre at $A$. We define
$\mathcal{N}_{<n}(f)$ to be the set of points 
in $\mathcal{N}_+(f)$ 
whose $z$-coordinate is smaller than $n$. A compact face ${\Gamma}$ of $\mathcal{N}_+(f)$ will be  called \emph{$\nu$-proper}
if  $A\in{\Gamma}$, $A\ne{\Gamma}$ and ${\Gamma}\setminus\{A\}\subset\mathcal{N}_{<n}(f)$. We define 
$\mathcal{N}^0(f)$ to be the set of all compact faces of $\pi_A(\mathcal{N}_{<n}(f))$.
\end{definition}

\begin{remark}
For a $\nu$-proper face $\Gamma$ of $\mathcal{N}(f)$, $\pi_A(\Gamma)$ is a compact 
face of $\mathcal{N}^0(f)$ of dimension one less than the dimension of $\Gamma$. 
In particular, $\mathcal{N}^0(f)$ consists of exactly one vertex if and only if 
$\mathcal{N}(f)$ has only one $\nu$-proper compact face and this face is $1$-dimensional.
\end{remark}

\begin{definition}\label{suitable}
A regular system of parameters  $x_1,\ldots,x_d,z$
of the local regular ring $\bk[[\bdx,z]]$ is called a
 \emph{suitable system of coordinates}  for a  polynomial $f \in \bk[[\bdx]][z]$
if either 
\begin{itemize}
\smallbreak\item $\mathcal{N}^0(f)$ is void, or
\smallbreak\item   $\mathcal{N}^0(f)$ has more than one vertex, or 
\smallbreak\item The set $\mathcal{N}^0(f)$ consists of exactly one vertex and the corresponding $\nu$-proper $1$-dimensional
compact face $\Gamma$ cannot be eliminated (Definition~\ref{eliminated}).
\end{itemize}
\end{definition}

\begin{remark}\label{suitable-change}
The condition to be a suitable system of coordinates for a polynomial $f\in \bk[[\bdx]][z]$ 
is weaker than the conditions of \emph{good coordinates} in \cite{aclm:05} because
it involves exactly one 
$\nu$-proper $1$-dimensional
compact face $\Gamma$ of the Newton diagram $\mathcal{N}_+(f)$, see also~\S\ref{P-good-coor-sec}.
\end{remark}

\begin{definition}\label{nu-quo}
Let $f(\bdx,z): \in \bk[[\bdx]][z]$ 
be a polynomial and assume we are as in Notation \ref{inicial} and Definition~\ref{inicial2}. 
We say that $f$ is a \emph{$\nu$-quasi-ordinary} polynomial if 
\begin{itemize}
\smallbreak\item $\mathcal{N}(f)$ has only one $\nu$-proper face $\Gamma_1$,
\smallbreak\item $\Gamma_1$ is a $1$-dimensional face and 
\smallbreak\item $\Gamma_1$ cannot be eliminated. 
\end{itemize}
The polynomial $f_{\Gamma_1 }$ is called the \emph{initial form} of $f$.
\end{definition}

The class of \emph{$\nu$-quasi-ordinary} singularities was introduced by
H. Hironaka  in \cite{hr:74} where he proved that 
quasi-ordinary polynomials are indeed  $\nu$-quasi-ordinary polynomials. 
In comparison with the condition to be a quasi-ordinary polynomial, 
$\nu$-quasi-ordinary is a very mild condition. 
However, $\nu$-quasi-ordinary have interesting properties and some of them will be discussed here.

We consider  $f$ to be a $\nu$-quasi-ordinary polynomial with its $\nu$-proper face $\Gamma_1$
and  keep the above notations.
The initial form  $f_{\Gamma_1 }$ of $f$ can be  written as follows:
\begin{equation}
\label{eq-fgamma}
f_{\Gamma_1 }=
a_{\Gamma_1 }\bdx^{\mathbf{n}}z^{n_{d+1}}\prod _{j=1}^{k}(z^p-\mu _j 
\bdx^{\mathbf{q}})^{m_j},
\end{equation}
where factors are  irreducible in $\bk[[\bdx]][z]$, 
i.e. $\gcd (\mathbf{q},p):=
\gcd (q_1,\dots , q_d,p)=1$, $\mu_j\in \bk^*$ with $\mu_j\ne \mu_i$ if $i\neq j$, and $a_{\Gamma_1 }\in \bk^*$
(see Proposition~\ref{factor-cara-2}).

\begin{definition}
A polynomial $f(\bdx,z)\in \bk[[\bdx]][z]$ is called \emph{elementary} 
if its Newton polyhedron $\mathcal{N}(f)$ consists of
 only one compact face which is a line segment 
$[(\mathbf{0},n),(\mathbf{r},0)]\subset \mathbb{R}^{d+1}$. 

More generally, if 
$\Gamma $ is a compact one-dimensional face of the Newton diagram $\mathcal{N}_+(f)$ of a
  polynomial $f(\bdx,z)\in \bk[[\bdx]][z]$, 
we say that the polynomial $g$ is \emph{$\Gamma $-elementary} if  its Newton polyhedron
$\mathcal{N}(g)$ consists of one compact face $\widetilde  \Gamma$
which is a translation of the face $\Gamma $. 
We denote the initial form of $g$ by $\In (g):=g_{\widetilde \Gamma }\in \bk[\bdx,z]$.
\end{definition}

In \cite[Theorem 1.5]{aclm:09} the following factorization theorem was proved.

\begin{fthm}\label{teorema1}
Let  $f\in \bk[[\bdx]][z]$ be a $\nu$-quasi-ordinary polynomial with its $\nu$-proper face $\Gamma_1 $
such that $f_{\Gamma_1}$ is as in{\rm~\eqref{eq-fgamma}}. 
Then there exist $k$ different
$\Gamma_1$-elementary polynomials $g_{\Gamma_1,j}\in \bk[[\bdx]][z]$, for $1\leq j \leq k$, which divide~$f$ and such that 
$\In (g_{\Gamma_1,j})=
(z^p-\mu _j \bdx^{\mathbf{q}})^{m_j}$, that is 
\begin{equation}\label{teorema1-factor}
 f=h G^{\Gamma_1}, \,\, \textit{with  }\,\, G^{\Gamma_1}:=\prod_{j=1}^k g_{\Gamma_1,j},\,\, \textit{ for some }\,\,  h\in  \bk[[\bdx]][z].
\end{equation}
\end{fthm}

Factorization Theorem~\ref{teorema1} has the following consequences.

\begin{cor}
Let $f$ be an irreducible  polynomial which is $\nu$-quasi-ordinary. 
Then it is elementary and $\In (f)=f_0^k$,
where $f_0 \in \bk[\bdx,z]$ is irreducible. 
 \end{cor}
 
 \begin{cor}
Let $f$ be an elementary  polynomial such that
$\In (f)=\psi_1\cdot\ldots\cdot \psi_r$ is a factorization of $\In (f)$ where the factors
are pairwise coprime. Then we can decompose $f=f_1\cdot\ldots\cdot f_r$ where, for all $i=1,\dots, r$, the series 
$f_i$ is an elementary polynomial
such that $\In(f_i)=\psi _i$.
\end{cor}

\begin{cor}\label{cor-factor-gamma}
Let $f$ be a $\nu$-quasi-ordinary polynomial.
Then
$f=\bdx^{a}z^b f^{\Gamma_1 } g$
where $f^{\Gamma_1}$ is $\Gamma_1$-elementary, and $\In (f^{\Gamma_1})=f_{\Gamma_1}$.
\end{cor}

\subsection{On Newton maps and the Newton proccess }\label{sec-newton}
\mbox{}

We need to introduce more notations  in order to define \emph{Newton maps}. For every $1\leq i\leq d$, let 
\begin{equation}
\label{eq-gcd}
c_i:=\gcd (p,q_i),\quad
p_i:=\frac{p}{c_i},\text{ and } q'_i:=\frac{q_i}{c_i}.
\end{equation}
We will consider the following three maps. 
The first map depends both on the face $\Gamma_1$ and on the root $\mu_j$, see \eqref{eq-fgamma}.
Let 
 $(\mathbf{u},u) \in \bn^{d+1}$ be integers such that
 $1+\mathbf{u}\cdot\mathbf{q} =u p$. Let  $\delta _{\Gamma_1, j}$ be the map
$$
\begin{matrix} 
& \delta _{\Gamma_1, j}:&\bk[[\bdx]][z]& \longrightarrow & \bk[[\bdx]][z]\hfill\\
&&h(\bdx,z)&\mapsto  &h(\mu_j^{u_1} x_1,\dots ,\mu_j^{u_d} x_d,
z).
\end{matrix}
$$
Then
$$
\delta _{\Gamma_1, j}(z^p-\mu _j 
\bdx^{\mathbf{q}})=z^p-\mu_j^{up}
\bdx^{\mathbf{q}}.
$$
The second map $\epsilon _{\Gamma_1}$ depends only on the face $\Gamma_1$:
$$
\begin{matrix}
&\epsilon _{\Gamma_1 }:&\bk[[\bdx]][z]& \longrightarrow & \bk[[\bdy]][z_1]\hfill\\
&&h(\bdx,z)&\mapsto  &h(y_1^{p_1},\dots ,y_d^{p_d},
\bdy^{\mathbf{q'}}z_1).
\end{matrix}
$$
It is easily seen that 
$$
(\epsilon _{\Gamma_1}\circ\delta _{\Gamma_1, j})(z^p-\mu _j
\bdx^{\mathbf{q}})
=
\bdy^{p \mathbf{q'}} (z_1^p-\mu _j^{u p}).
$$

Now the third map $\tau_j$ will depend on the root $\mu _j$: 
$$
\begin{matrix}
&\tau _j:&\bk[[\bdy]][z_1]& \longrightarrow & \bk[[\bdy]][z_2]\hfill\\
&&h(\bdy,z_1)&\mapsto  &h(\bdy,z_2+\mu_j^u).
\end{matrix}
$$
Note that:
$$
(\tau _j\circ\epsilon _{\Gamma_1}\circ\delta _{\Gamma_1, j})(z^p-\mu _j \bdx^{\mathbf{q}})
=
\bdy^{\mathbf{p*q}}((z_2+\mu_j^u)^p-\mu _j^{u p})
$$ 
and the last factor is of order one in $z_2$. 
Note that $p_i q_i= p q'_i$ for $1\leq i\leq d$, i.e, $\mathbf{p*q}=p\mathbf{q'}$.

\begin{definition}\label{newton-map}
The \emph{Newton map} $\sigma _{\Gamma_1 ,j}$ associated with $\Gamma_1$ and the root $\mu_j$
is the composition map $
\sigma _{\Gamma_1 ,j}:= \tau _j\circ\epsilon _{\Gamma_1}\circ\delta _{\Gamma_1, j}:
\bk[[\bdx]][z] \to  \bk[[\bdy]][z_2]$ .

\end{definition}

Also the following statement was proved in \cite[Lemma 2.2]{aclm:09}.

\begin{lemma}\label{total-transf}
After the Newton map $\sigma _{\Gamma_1, j}$,  the total transform  $f_{\Gamma_1, j}:=\sigma _{\Gamma_1, j}(f)\in \bk[[\bdy]][z_2]$
of the polynomial $f$ can be written as 
$$
f_{\Gamma_1, j}(\bdy,z_2)=
\bdy^{m\mathbf{p*q}+\mathbf{n*p}} f_{1,\Gamma_1, j}(\bdy,z_2 ),
$$ 
where $f_{1,\Gamma_1, j}(\mathbf{0},z)$ is    regular 
of order $m_j$, and $m:=\sum_{j=1}^k m_j$.

Moreover by chain rule $$\frac{\partial(f\circ \sigma _{\Gamma_1, j})}{\partial z_2}(\bdy,z_2 ) =
\bdy^\mathbf{q'} \left(\frac{\partial f}{\partial z}\circ 
\sigma _{\Gamma_1, j}\right)(\bdy,z_2 ).$$
\end{lemma}

\subsubsection{Newton's process associated with the $\nu$-proper face $\Gamma_1 $}\label{new-pro-primera}
\mbox{}
\vspace{.5cm}

We can perform a change of variables of the type $z\mapsto z+h(\bdx)$ in order to have
a suitable system of coordinates for $f_{1,\Gamma_1, j}$; if $f_{1,\Gamma_1, j}$ is again $\nu$-quasi-ordinary
one can iterate the process until one gets either a monomial times a unit or a non $\nu$-quasi-ordinary polynomial.
This process is called Newton's process; note that $m_j<n$ because the face $\Gamma_1$ cannot be eliminated.

\subsubsection{Factorization and the Newton process associated with other compact 1-dimensional faces of the
Newton polyhedron}\label{new-pro-otras}
\mbox{}
\vspace{.5cm}

Since $f$ is a $\nu$-quasi-ordinary polynomial with $\nu$-proper face $\Gamma_1 $
then $\Gamma_1$ is a 
$1$-dimensional face of its Newton diagram $\mathcal{N}_+(f)$. We   
assume that  $\Gamma_1$ is the segment $[A,A_1]$. 
If the $z$-coordinate of $A_1$ is $n^1>0$, we denote
by $\mathcal{N}_{<n^1}(f)$ the set of points 
in $\mathcal{N}_+(f)$ 
whose $z$-coordinate is smaller than $n^1$.
Let $\pi _{A_1}:\mathcal{N}_{<n^1}(f)\setminus A_1 \to \mathbb{Q}^d $ be the projection into with center $A_1$
and  let
$\mathcal{N}^{0,1}(f)$ be the convex hull of the image by $\pi_{A_1}$  of $\mathcal{N}_{<n^1}(f)$. 
If $\mathcal{N} ^{0,1}(f)$ has only one vertex then there is another face $\Gamma_2$ 
of the Newton diagram $\mathcal{N_+}(f)$  which is of dimension $1$. 
We go further on this construction with $\Gamma_1=[A,A_1],\Gamma_2=[A_1,A_2],\cdots, \Gamma_s=[A_s,A_{s+1}]$ until
one of the following cases arises:
\begin{enumerate}
\enet{\rm(NW\arabic{enumi})} 
\item\label{nw1} The $z$-coordinate $n^{s+1}$ of $A_{s+1}$ is zero.
\item\label{nw2} $\mathcal{N}^{0,{s+1}}(f)=\emptyset$.
\item\label{nw3} $\mathcal{N}^{0,{s+1}}(f)$ has more than one vertex.
\end{enumerate}
Moreover $\Gamma_1 \cup \Gamma_2 \cup \cdots \cup \Gamma_{s}$ is a 
\emph{monotone polygonal path} in $\mathcal{N}(f)$.

\begin{lemma}
The Newton polyhedron of $f$ is not a 
\emph{monotone polygonal path} if and only if $\mathcal{N}^{0,{s+1}}(f)$ has more than one vertex, i.e., \ref{nw3} arises. 
\end{lemma}

For every  edge $\Gamma_\ell$ of $ \Gamma_2, \cdots, \Gamma_{s}$,
 the initial form  $f_{\Gamma_\ell }$ of $f$ can be written as in (\ref{eq-fgamma})
as follows:
\begin{equation}
\label{eq-fgamma-ell}
f_{\Gamma_\ell }=
a_{\Gamma_\ell }\bdx^{\mathbf{n^\ell}}z^{n^\ell_{d+1}}\prod _{j=1}^{k^\ell}(z^{p_\ell}-\mu _j^\ell 
\bdx^{\mathbf{q^\ell}})^{m_j^\ell},
\end{equation}
the factors being irreducible in $\bk[[\bdx]][z]$, 
i.e. $\gcd (
\mathbf{q^\ell},p_\ell)=1$, $\mu_j^\ell\in \bk^*$ with $\mu_j^\ell\ne \mu_i^\ell$ and $a_{\Gamma_\ell }\in \bk^*$.
For each root $\mu _j^\ell$
of its face polynomial $f_{\Gamma_\ell }$ one applies the corresponding Newton map $\sigma _{\Gamma_\ell, \mu _j^\ell}$. 
At each step, we encode the information given by the corresponding Newton diagram.
The process stops because the $z$-degree decreases since we are in a suitable system of coordinates. 
In next section Newton trees of a polynomial $f(\bdx,z)\in \bk[[\bdx]][z]$ are constructed by recursion 
on the number of steps of subsections~\S\ref{new-pro-primera} and~\S\ref{new-pro-otras}.

One can also apply recursively Factorization Theorem \ref{teorema1} 
to $h$, see \eqref{teorema1-factor}, to get  $k^\ell$ different
$\Gamma_\ell$-elementary polynomials $g_{\Gamma_\ell,j}\in \bk[[\bdx]][z]$, for $1\leq j \leq k^\ell$, 
such that $G^{\Gamma_\ell}:=\prod_{j=1}^{k^\ell} g_{\Gamma_\ell,j}$
divides~$f$, that is 
\begin{equation}\label{f-factor}
 f=H(\bdx,z)\prod_{t=1}^s G^{\Gamma_t}, \,\,\, \text{ for some }\,\,  H\in  \bk[[\bdx]][z].
\end{equation}

\section{The Newton tree of a polynomial}\label{sec-trees}

\subsection{Construction of the Newton tree}\label{constr-trees}

\mbox{}

Let $f(\bdx,z)\in \bk[[\bdx]][z]$ be a polynomial
and we assume that we are as in Notation \ref{inicial}. In this section we associate with 
$f(\bdx,z)$ a tree $\mathcal{T}(f)$ called 
\emph{Newton tree of $f$} because its first steps are  built using both its Newton diagram $\mathcal{N_+} (f)$ and 
the set $\mathcal{N}^0(f)$ of compact faces of $\pi(\mathcal{N}_{<n}(f))$. Further steps will be
 based on the Newton
process associated with  $f$ (see subsections \ref{new-pro-primera} and \ref{new-pro-otras}).

For a given~$f$, we are going to associate a tree $\mathcal{T}_\mathcal{N}(f)$,
called \emph{vertical tree}.
First, if $f$ is not in a suitable system of coordinates (see Definition~\ref{suitable}),
we perform a change of variables such that it is the case; in order to simplify
the notations we denote again by~$f$ the resulting polynomial. We keep the notations
of Definition~\ref{inicial2}. The tree $\mathcal{T}_\mathcal{N}(f)$ is
built using its Newton diagram $\mathcal{N}_+ (f)$.

We distinguish three cases.
\begin{caso0}\label{caso1a}
The Newton polyhedron $\mathcal{N} (f)$ consists only in one vertex (see Remark \ref{1-vertex}). 
Then the Newton tree $\mathcal{T}_{\mathcal{N}}(f)$ of $f$ is given in Figure~\ref{Q01}.
\end{caso0}

\begin{figure}[ht]
\centering
\subfigure[]{
\includegraphics[scale=1]{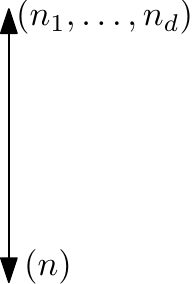}
\label{Q01}
}
\hfil
\subfigure[]{
\includegraphics[scale=1]{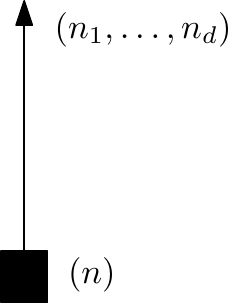}
\label{Q02}
}
\caption{}
\label{fig0}
\end{figure}

\begin{caso0}\label{caso2a}
The set $\mathcal{N}^0(f)$ of all compact faces of $\pi(\mathcal{N}_{<n}(f))$ has more than one vertex. 
Then the Newton tree $\mathcal{T}_{\mathcal{N}}(f)$ of $f$ is given in Figure~\ref{Q02}.  
\end{caso0}

\begin{caso0}\label{caso3a}
The set $\mathcal{N}^0(f)$ has exactly one 
vertex. 
Since we are in a suitable system of coordinates for $f$, 
the face $\Gamma_1$ cannot be eliminated, 
(i.e. $f$ is a $\nu$-quasi-ordinary polynomial). 
In such a case we are as in subsection~\ref{new-pro-otras}, and there exist 
$s$ compact edges $\Gamma_1:=[A,A_1],\Gamma_2:=[A_1,A_2],\cdots, \Gamma_s:=[A_s,A_{s+1}]$
of the Newton polyhedron $\mathcal{N}(f)$ 
until one of the cases~\ref{nw1},~\ref{nw2} or~\ref{nw3} happen.

Furthermore $\Gamma_1 \cup \Gamma_2 \cup \cdots \cup \Gamma_{s}$ is a 
\emph{monotone polygonal path} in $\mathcal{N}(f)$.
With this monotone polygonal path we associate  a decorated \textbf{vertical} graph $\mathcal{T}_{\mathcal{N}}(f)$ 
(which depends on $\mathcal{N}(f)$ in a suitable system of
coordinates for $f$). With each compact 1-dimensional face $\Gamma_\ell$ of the polygonal path,
we associate a vertex $v_\ell$. If two compact faces intersect at one point we draw a vertical edge from one vertex to the other.
Thus these vertices are drawn on a vertical line by the increasing order of the slopes, i.e.
$v_1, v_2,\ldots, v_s$
from above to below in order. Decorations of this \textbf{vertical line} $\mathcal{T}_{\mathcal{N}}(f)$ are as follows:

\begin{itemize}
\item On the top of the vertical line we add an arrow-head decorated with $\bdn:=(n_1,\cdots,n_d)$. 
\item On the bottom of the vertical line, we add: 
\begin{itemize}
\item an arrow-head decorated with $(n^{s+1})$ in cases~\ref{nw1} ($n^{s+1}=0$) and~\ref{nw2};
\item a black box decorated with $(n^{s+1})$ in case~\ref{nw3}.
\end{itemize}
\item for an edge $\Gamma_\ell$  which is defined in 
coordinates $(\balpha,\beta):=(\alpha_1,\ldots,\alpha_d,\beta)$ of $\mathbb{R}^{d+1}$ by the intersection of $d$ hyperplanes of equations
\begin{equation}\label{eq-edge}
p_\ell \alpha _k+q_k^\ell \beta =N_k^\ell,\,\qquad 1\leq k\leq d,\,\qquad \gcd (\mathbf{q}^\ell,p_\ell)=1,
\end{equation}
the corresponding vertex $v_\ell$ support the following decorations:
\begin{itemize}
\item The vertex itself is decorated with  $((N_1^\ell,\cdots, N_d^\ell))$.
\item The lower edge is decorated near $v_\ell$ with $p_\ell$.
\item The upper edge is decorated near $v_\ell$ with $(q_1^\ell,\cdots,q_d^\ell)$.
\end{itemize}
\end{itemize}
\end{caso0}

We describe now the construction of $\mathcal{T}(f)$. Recall that we assume
that $f$ is in suitable coordinates.

\begin{paso}
If $f$ is in either Case~\ref{caso1a} or~\ref{caso2a} then $\mathcal{T}(f):=\mathcal{T}_{\mathcal{N}}(f)$.
If $f$ is in Case~\ref{caso3a} we continue the process.
\end{paso}

\begin{paso}
For every 1-dimensional face $f_{\Gamma_\ell }$ and for each root $\mu _j^\ell$
of its face polynomial $f_{\Gamma_\ell }$ one applies 
the corresponding Newton map $\sigma _{\Gamma_\ell, \mu _j^\ell}$
to $f$ and get a polynomial $f_{\Gamma_\ell, j}$;
we perform a change of coordinates to be in suitable coordinates.
\end{paso}

\begin{paso}\label{paso3}
The tree $\mathcal{T}(f)$ will be obtained by gluing in a suitable way the
tree $\mathcal{T}_{\mathcal{N}}(f)$ and $\mathcal{T}(f_{\Gamma_\ell, j})$
which may be assumed constructed recursively. 
The tree $\mathcal{T}(f_{\Gamma_\ell, j})$ will be attached to the vertex $v_\ell$
of $\mathcal{T}_{\mathcal{N}}(f)$ by a \textbf{horizontal edge} which links $v_\ell$ to the top vertex of
 $\mathcal{T}_{\mathcal{N}}(f_{\Gamma_\ell, j})$ where the top arrow has been deleted. 
\end{paso}
 
Now to decorate the tree $\mathcal{T}(f)$ we need 
more definitions. Let $v$ be a vertex on $\mathcal{T}(f)$. If $v$ is on the first left vertical line, 
we say that $v$ has no preceding vertex. 
If $v$ is not on the first left vertical line, the vertical line on which $v$ lies, 
is linked by an horizontal edge ending on a vertex $v_1$, to a vertical line. 
Then $v_1$ is said to be the preceding vertex of $v$. Note that the path between $v_1$ and $v$ can have many vertical edges, but has exactly one horizontal edge. Now denote by $\mathcal{S}(v)=\{v_i,v_{i-1},\cdots, v_1,v_0=v\}$ where $v_j$ is the preceding vertex of $v_{j-1}$ for $j=1,\cdots,i$ and $v_i$ has no preceding vertex. We say that $\mathcal{S}(v) \setminus \{v\}$ is the set of preceding vertices of $v$. {Now in the contruction of the Newton tree $\mathcal{T}(f)$, we glue on 
$\mathcal{T}_{\mathcal{N}}(f)$ at $v_i$ a tree $\mathcal{T}(f_{\Gamma_i, l})$ 
where $\mathcal{S}(v)=\{v_{i-1},\cdots, v_1,v_0=v\}$. 
Assume the decorations of the edges attached to $v_i$ on $\mathcal{T}_{\mathcal{N}}(f)$ are 
$((q_1^i,\cdots, q_d^i),p_i)$.

We are going to decorate the edges attached to $v_0$ in $\mathcal{T}(f)$ with $((Q_1^0,\cdots,Q_d^0),p_0)$
following these rules:
\begin{itemize}
\item The decoration $p_0$ coincides with the corresponding decoration in the vertical tree containing~$v_0$.
\item Let us assume that the decorations of $v_j$ on $\mathcal{T}(f_{\Gamma_i, l})$ 
are $((Q_1^{j,i},\cdots,Q_d^{j,i}),p_j)$ for $j=0,\cdots,i-1$.
Then
\begin{equation}\label{eq-newton-tree-Q}
Q_k^0:=\frac{p_i q_k^{i} p_{i-1}^2\cdot\ldots\cdot p_1^2 p_0}
{\gcd(p_i,q_k^i)\cdot\gcd(p_{i-1},Q_k^{i-1,i})\cdot\ldots\cdot\gcd(p_1,Q_k^{1,i})}+Q_k^{0,i},\,\,\, k\in \{1,\cdots,d\}.
\end{equation}
Note that, in particular, $Q_k^i=q_k^i$.
\end{itemize}

It is useful to have another way for computing these decorations, see~\cite[(5.1)]{aclm:09}. Let us assume with the above notations that
the decorations $((Q_1^j,\cdots,Q_d^j),p_j)$ for $\mathcal{T}(f)$ have been defined for $v_j$, $j=0,1,\dots,i-1$.
Recall that $Q_k^i=q_k^i$. Let us denote $((q_1^0,\cdots, q_d^0),p_0)$ the decorations of $v_0$ in its vertical tree.

\begin{lemma}\label{lema-newton-tree-Q}
With the above notations
$$
Q_k^{0}=\frac{p_1 Q_k^1 p_0}{\gcd(p_1,Q_k^1)}+q_k^{0}.
$$
\end{lemma}

\begin{proof}
Let us assume first that $i=1$. In this case $Q_k^{0,1}=q_k^0$ since 
$\mathcal{T}(f_{\Gamma_1, l})=\mathcal{T}_{\mathcal{N}}(f_{\Gamma_1, l})$ and  $Q_k^1=q_k^1$
since $i=1$. The result follows from a direct substitution in~\eqref{eq-newton-tree-Q}.

If $i>1$ then, by induction hypothesis, we have
$
Q_k^{0,i}=\frac{p_1 Q_k^{1,i} p_0}{\gcd(p_1,Q_k^{1,i})}+q_k^{0}
$.
Hence,
\begin{equation*}
\begin{split}
Q_k^0=&\frac{p_i q_k^{i} p_{i-1}^2\cdot\ldots\cdot p_1^2 p_0}
{\gcd(p_i,q_k^i)\cdot\gcd(p_{i-1},Q_k^{i-1,i})\cdot\ldots\cdot\gcd(p_1,Q_k^{1,i})}+Q_k^{0,i}\\
=& \frac{p_i q_k^{i} p_{i-1}^2\cdot\ldots\cdot p_1^2 p_0}
{\gcd(p_i,q_k^i)\cdot\gcd(p_{i-1},Q_k^{i-1,i})\cdot\ldots\cdot\gcd(p_1,Q_k^{1,i})}
+\frac{p_1 Q_k^{1,i} p_0}{\gcd(p_1,Q_k^{1,i})}+q_k^{0}\\
=&\left(\frac{p_i q_k^{i} p_{i-1}^2\cdot\ldots\cdot p_2^2 p_1}
{\gcd(p_i,q_k^i)\cdot\gcd(p_{i-1},Q_k^{i-1,i})\cdot\ldots\cdot\gcd(p_1,Q_k^{1,i})}
+Q_k^{1,i}\right)\frac{p_1 p_0}{\gcd(p_1,Q_k^{1,i})}+q_k^{0}.
\end{split}
\end{equation*}
Using \eqref{eq-newton-tree-Q}
$$
Q_k^1=\frac{p_i q_k^{i} p_{i-1}^2\cdot\ldots\cdot p_2^2 p_1}
{\gcd(p_i,q_k^i)\cdot\gcd(p_{i-1},Q_k^{i-1,i})\cdot\ldots\cdot\gcd(p_1,Q_k^{1,i})}
+Q_k^{1,i}
$$
and the result follows.
\end{proof}

\begin{remark}\label{vert-hori}
By construction the Newton tree $\mathcal{T}(f)$ has  \textbf{vertical parts} and \textbf{horizontal parts}:
vertical parts correspond to Newton diagrams of total transforms by Newton maps and horizontal edges 
are edges used for connecting vertical parts. 
\end{remark}

\begin{definition}\label{ends}
An \emph{end} of the Newton tree $\mathcal{T}(f)$ is either an arrow-head or a  black box. The arrow-heads or black boxes
decorated with~$(0)$ will be called \emph{dead ends}.  
\end{definition}

\begin{remark}\label{end-arrows}
By construction black boxes in the Newton tree $\mathcal{T}(f)$ appear if and only if 
we reach at some step Case~\ref{caso2a}. 
This means that the Newton tree $\mathcal{T}(f)$ ends with arrow-heads if and only if 
we never reach Case~\ref{caso2a}.
\end{remark}

\begin{definition}\label{gl-data}
For every vertex $v_j$ of the Newton tree $\mathcal{T}(f)$, the
\emph{global numerical data} of the vertex $v_j$ is 
$(\mathbf{N}_{v_j},p_j)$ where $\mathbf{N}_{v_j}:=(N^j_{1},\cdots ,N^j_{d})$.
\end{definition}
\begin{remark}

Global numerical data 
of the first vertex satisfy the following property. If 
$\sigma$ is a Newton map associated with this vertex, then 
$$
\sigma (f)=\bdy^{\mathbf{\frac{N}{c}}}
f_1(\bdy, z_1)
$$
where no $y_i$ divides $f_1$ and $c_i=\gcd(p,q_i)$ as in \eqref{eq-gcd}.
\end{remark}

\begin{definition}\label{lc-data}
The numerical data
$((Q^j_{1},\cdots,Q^j_{d}),p_j)$
associated with each vertex $v_j$ of the Newton tree $\mathcal{T}(f)$ 
are called the \emph{local numerical data} of the vertex $v_j$. 
The \emph{gcd of the vertex}~$v_j$ is $\mathbf{c}:=(\gcd(Q^j_{1},p_j),\dots,\gcd(Q^j_{d},p_j))$.
\end{definition}

\begin{remark}\label{rem-growth}
It is easy to show that, 
if $((Q_1,\cdots,Q_{d}),p)$ are the local numerical data
of a vertex~$v$ and $((Q'_{1},\cdots,Q'_{d}),p_j)$ are the local numerical data 
of a vertex $v'$ such that $v$ is the preceding vertex of $v'$.
Then, applying Lemma~\ref{lema-newton-tree-Q}, the following strict inequalities hold:
$$
Q'_i>\frac{p Q_i }{\gcd(p,Q_i)} p_j,  \,\,\text{for}\,\,
i=1,\cdots, d, \,\,\text{and}\,\, j=1,\cdots, r.
$$
This condition is called \emph{the growth condition on the local numerical data}.
\end{remark}

\begin{remark}

In the case for  $d=1$,  and if we forget about vertical and horizontal edges the Newton trees 
decorated with local numerical data are the Eisenbud-Neumann diagrams of the corresponding germ defined in \cite{en:85}.

\end{remark}

\begin{definition}
The \emph{valency} of a vertex in a Newton tree is the number of edges attached to the vertex. 
\end{definition}

\subsection{Comparable polynomials and coloured Newton trees}

\mbox{}

Let $f(\bdx,z), g(\bdx,z)\in \bk[[\bdx]][z]$ be polynomials. 
We assume that we are as in Notation~\ref{inicial} and Definition~\ref{inicial2}.

\begin{definition}\label{def-comp}
Two polynomials $f$ and $g$ in $\bk[[\bdx]][z]$ are called \emph{comparable} if 
their resultant $\res_z(f,g)(\bdx)$ of
$f$ and $g$ with respect to $z$ is equal to a monomial times a unit, that is,
$$
\res_z(f,g)(\bdx)=\bdx^\bdn u(\bdx), \,\,\textrm{ with  }\,\, \,\,\bdn\in{\mathbb{N}}^d,\,\,u(\bdx)\in \bk[[\bdx]],  \textrm{ and  }
\,\, u(\textbf{0} )\ne 0.
$$

\end{definition}

\subsubsection{Coloured Newton trees} 
\mbox{}
\vspace{.5cm}

\emph{Coloured Newton trees} are associated with the product of two
polynomials $f(\bdx,z), g(\bdx,z)\in \bk[[\bdx]][z]$. We assume $fg$ is in suitable coordinates. Take the Newton tree $\mathcal{T}(fg)$ 
of the product $f g=f(\bdx,z)\cdot  g(\bdx,z)$ 
and add two colours to it, say red and blue.
Blue colour  is associated with $f$ and red colour with $g$. 
The coloured Newton tree $\mathcal{T}(fg)$  can have blue parts, red parts and blue-red parts.

\begin{definition}\label{coloured}
We consider two polynomials 
$f(\bdx,z):=\bdx^\bdn  f_1(\bdx,z),g(\bdx,z):=\bdx^\bdm  g_1(\bdx,z)\in \bk[[\bdx]][z]$  such that 
 $f_1(\bdx,z)$ and $g_1(\bdx,z)$ are regular 
of order say $m,n\geq 0$. In such a case the product 
$$
f_1 g_1(\textbf{0},z)=a_0z^{m+n}+\cdots 
$$
satisfies $a_0 \neq 0$. In this case  $A:=(\mathbf{0},m+n)\in \mathcal {N}_+ (f _1g_1)$ and
 let $\pi$ be the projection into
$\bq ^d$ with centre $A$ as before and we consider $\mathcal {N}_0 (f_1 g_1)$ as in Definition~\ref{inicial2}. 
We consider the Newton diagram  
$\mathcal {N}_+(f g)=\mathcal {N}_+(\bdx^{\bdn+\bdm})+\mathcal {N}_+(f_1 g_1)$ 
which is used to construct the first vertical part $\mathcal{T}_\mathcal{N}(fg)$ of the Newton tree $\mathcal{T}(fg)$ 
of $f g$. Three possible cases may arise either 

\begin{enumerate}
 \item $\mathcal{T}_\mathcal{N}(fg)$ is 
blue coloured if $\deg g_1(\textbf{0},z)=0$ or, 
\item $\mathcal{T}_\mathcal{N}(fg)$ is 
red coloured if $\deg f_1(\textbf{0},z)=0$ or, 
\item $\mathcal{T}_\mathcal{N}(fg)$ is bi-coloured blue-red otherwise.
\end{enumerate}

We apply the same rule for every  steps of the Newton process.
In particular every vertical line in the Newton tree $\mathcal{T}(fg)$  of $f g$ has the same  (bi)color.
Bicoloured vertices of the bicoloured Newton tree $\mathcal{T}(fg)$ will be called \emph{common vertices} of $f$ and $g$.
\end{definition}

\begin{example}

The bi-coloured Newton tree $\mathcal{T}(fg)$ of $f g$ where $f=z^2-x^3$ and $g=z^3-x^2$ is as in Figure~\ref{QOnov1} where all vertical lines are bi-coloured, the above horizontal line is red and the below one is blue.

\begin{figure}[ht]
\begin{center}
\includegraphics[scale=1]{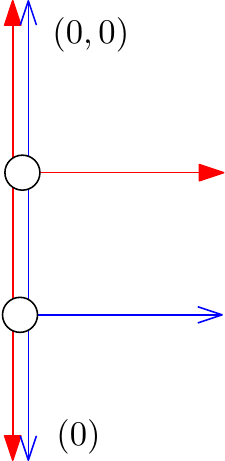}
\caption{}
\label{QOnov1}
\end{center}
\end{figure}
\end{example}

\begin{definition}\label{separated}
Two  polynomials $f$ and $g$ in $\bk[[\bdx]][z]$ 
are called \emph{separated} if there exists a suitable system of coordinates for $fg$ such that all  ends which are not dead-ends of the bi-coloured Newton tree 
$\mathcal{T}(fg)$  are either red coloured or blue coloured, see Definition~\ref{ends}. 
\end{definition}

\begin{remark}
If $d=1$, $f$ and $g$ are separated if they do not share a common component.
\end{remark}

\begin{definition}\label{separated-vertex}
Let $f,g\in\bk[[\bdx]][z]$ such that the system of coordinates is suitable for $fg$.
Let $v$ a vertex of the Newton tree 
$\mathcal{T}(fg)$ corresponding to an edge $\Gamma$. Let $\hat {f}_\Gamma$ obtained from $f_\Gamma$ by taking away the powers of $x_i$ and $z$.
We say that $f$ and $g$ \emph{separate at $v$} if $\gcd(\hat{f}_\Gamma,\hat{g}_\Gamma)$ is not equal to 
either $\hat{f}_\Gamma$ or $\hat{g}_\Gamma$.
The \emph{order of separation} of $f$ at $v$ is the 
$z$-degree of $\frac{\hat{f}_\Gamma}{\gcd(\hat{f}_\Gamma,\hat{g}_\Gamma)}$.
\end{definition}

\begin{example}
We illustrate the two kinds of separation of $f$ and $g$. In Figure~\ref{QO44nna}
after two Newton maps the total transform of $fg$, see Lemma~\ref{total-transf}, has
 a Newton polyhedron with two different compact 1-dimensional faces 
and each of them corresponds either to (the total transform of) $f$ and the other to (the total transform of) 
$g$. In Figure~\ref{QO45nna} they have different polynomials on the same
face, see \eqref{eq-fgamma}.

\begin{figure}[ht]
\centering
\subfigure[]{
\includegraphics[scale=1]{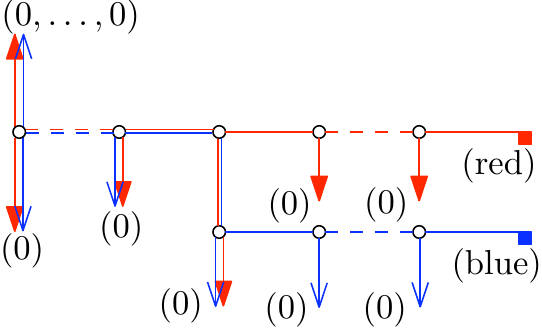}
\label{QO44nna}
}
\hfil
\subfigure[]{
\includegraphics[scale=1]{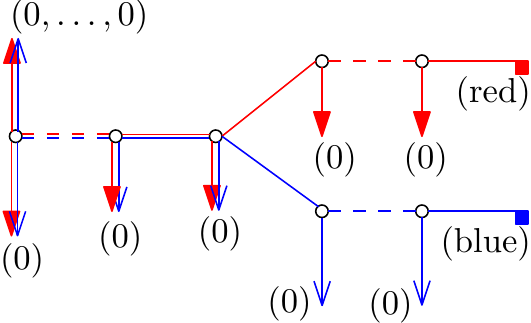}
\label{QO45nna}
}
\caption{}
\label{fig4}
\end{figure}
\end{example}

\begin{remark}\label{vertical-separated}
If two  polynomials $f,g\in\bk[[\bdx]][z]$ 
are \emph{separated} then there is at least a vertical line in  the bi-coloured Newton tree $\mathcal{T}(fg)$ 
which corresponds 
to a Newton polygon of a step of its Newton process, after which its Newton tree $\mathcal{T}(fg)$ is not bi-coloured,
e.g see Figures~\ref{QOnov1},~\ref{QO44nna} or~\ref{QO45nna}.   
\end{remark}

\begin{remark}
Let $f,g\in\bk[[\bdx]][z]$ be  two polynomials. Assume there exists a suitable system of coordinates of each of them.  For a common vertex of the bicoloured Newton tree $\mathcal{T}(fg)$ 
its local numerical data is the same as its local numerical data as vertex of $\mathcal{T}(f)$ and the same as
its local numerical data as vertex of $\mathcal{T}(g)$.
\end{remark}

Next theorem is one of the main results in~\cite{aclm:09}.

\begin{theorem}[{\cite[Theorem 1.6]{aclm:09}}]\label{thm-sepcom}
If two  polynomials $f$ and $g$ in  $\bk[[\bdx]][z]$ are separated 
then they are comparable, i.e. their resultant $\res_z(f,g)=\bdx^{\balpha} \eta(\bdx)$
 with $\eta(\bdx)\in \bk[[\bdx]]$, 
$\balpha\in \bn^d$ and $\eta(\textbf{0})\ne 0$.
\end{theorem}

\subsection{Computation of resultants}
\mbox{}

Let $f$ and $g$ be two  Weierstrass  polynomials  which are  comparable  polynomials and
such that their Newton trees have only one end which is not a dead-end. The main result in this section is to show
that the resultant $\res_z(f,g)$
can be read from the coloured Newton tree of $f g$ decorated with its local numerical data. 
This result can be seen as a generalization of \cite[Corollaire 5]{pgp:00} 
where  P.~Gonz{\'a}lez-P{\'e}rez gave information about the Newton diagram of the
resultant of two quasi-ordinary hypersurfaces satisfying an
appropriate non-degeneracy  condition.

Let $f$ and $g$ be two  Weierstrass polynomials  in $\bk[[\bdx]][z]$
and asume we are as in Notation~\ref{inicial} and  the system of coordinates is suitable for both 
$f$ and $g$.
We consider the coloured Newton tree $\mathcal{T}(fg)$  of $f g$ decorated with its local numerical data. 
We recall the following result 
\cite{aclm:09} where it was shown that to compute their resultant $\res_z(f,g)$ 
one can use the factorization of total transforms of $f$ and $g$ 
after Newton process
even if they are not  factorizations of $f$ and $g$.

\begin{proposition}[{\cite[Proposition 5.11]{aclm:09}}]\label{prop-sep}
Let $f,g\in \bk[[\bdx]][z]$ be two  Weierstrass  polynomials in suitable coordinates which are comparable  and
such that their Newton trees $\mathcal{T}(f)$ and $\mathcal{T}(g)$ 
have only one end which is not a dead-end. Then, $\res_z(f,g)$
can be read from the coloured Newton tree $\mathcal{T}(fg)$
of $f g$ decorated with its local numerical data.

More precisely, for every $1\leq i\leq d$ the multiplicity of $x_i$ in $\res_z(f,g)$ is computed as follows:
\begin{enumerate}
\item Consider the path from the blue-end representing
$f$ to the red-end representing $g$.
\item Take the product of all the decorations that are
adjacent to the path (the $i^{\rm\text{th}}$ coordinate in the case of $d$-uples), including
decorations  of these  ends.
\item Multiply it by the the $i^{\text{th}}$ coordinate of the $\gcd$ of the vertices before the vertex where they
separate, see Definition{\rm~\ref{lc-data}}.
\end{enumerate}
\end{proposition}

\section{P-Good coordinates}\label{P-good-coor-sec}

Let $f(\bdx,z)\in \bk[[\bdx]][z]$ be a polynomial and we assume that we are as in Notation \ref{inicial} 
and  in a suitable system  of coordinates for $f$.
There is a system of good coordinates, introduced by P.~Gonz\'alez-P\'erez in \cite{go:01}
 and \cite[Lemma 3.2]{GMN:03}, that is called P-good coordinates. 

\begin{definition}
A polynomial $f(\bdx,z)\in \bk[[\bdx]][z]$ is in  \emph{P-good coordinates}
if 
\begin{enumerate}
\item its Newton polyhedron $\mathcal{N}(f)$  is a monotone polygonal path,
\item  if
there exists a face $\Gamma$ of $\mathcal{N}(f)$  whose line supporting $\Gamma$ is given by 
equations  $\{\alpha _i +q_i \beta =N_i$, $i=1,\cdots d\}$, then $\Gamma$ is unique and 
hits the hyperplane $\beta =0$,
\item if
there exists a face $\Gamma$ of $\mathcal{N}(f)$ whose line supporting $\Gamma$ is given by 
equations  $\{\alpha _i +q_i \beta =N_i$, $i=1,\cdots d\}$, then the corresponding polynomial 
$f_\Gamma$ is not of the form $a_{\Gamma}\bdx^{\mathbf{m}} (z-h(\bdx))^n $
with $h(\bdx) \in \bk[\bdx]$ and $a_{\Gamma}\in \bk^*$, see~(\ref{eq-fgamma}).
\end{enumerate}
\end{definition}

Let us assume that there is a suitable system of coordinates such that the Newton polyhedron of $f$ is a monotone polygonal path. 
Let $\mathcal{T}(f)$ be the Newton tree of $f$ in this suitable system of coordinates. 

\begin{proposition}\label{prop-p-coord}
Starting from  a suitable system of coordinates such that the Newton polyhedron of $f$ is a monotone polygonal path
we find a P-good system of coordinates for $f$ by a change of coordinates of the type $z\mapsto z+a\bdx ^{\mathbf{q}}$.
The Newton tree in this system of coordinates can be deduce from $\mathcal{T}(f)$ in a unique way.
\end{proposition}

\begin{proof}
Let $\mathcal{T}_{\mathcal{N}}(f)$ be the first vertical line of $\mathcal{T}(f)$. The vertices are denoted by $v_1, \cdots, v_n$ from top to bottom decorated with $(\mathbf{q},p)$.
\begin{caso1}
None of the decorations $p$ of the vertices of $\mathcal{T}_{\mathcal{N}}(f)$ is equal to 1
\end{caso1}

In this case, $f$ is already in P-good coordinates.

\begin{caso1} There is a decoration $p$ of some vertex which is equal to one.
\end{caso1}

Let $i$ be the smallest index such that $p_i=1$. We have two cases to consider.

\begin{subcaso}\label{subcaso1} $i\neq n$.
\end{subcaso}

In this case, $f$ is not in P-good coordinates. Assume the decorations of the vertex $v_i$ 
in $\mathcal{T}_{\mathcal{N}}(f)$ are
$(q_1^i,\cdots,q_d^i,p_i=1)$ and the face polynomial is 
$$
\bdx^{k'}z^k\prod_j (z-\mu_j\bdx^{\mathbf{q}^i})^{m_j}.
$$

Consider the change of coordinates 
$$
z=z'+a\bdx ^{\mathbf{q}^i},
$$
where $a\neq \mu_j$ for all $j$. 
Under this change of variables neither the faces $\Gamma_1,\cdots, \Gamma_{i-1}$ 
nor  their face polynomials do change. The face $\Gamma_i$ 
has the same equation but now the face polynomial is 
$$
\bdx^{k'}(z'+a\bdx^{\mathbf{q}^i})\prod_j (z-(\mu_j-a)\bdx^{\mathbf{q}^i})^{m_j}.
$$
The new face $\Gamma_i'$ hits the $\bdx$-hyperplane.
Then $f$ is in P-good coordinates.

Now we compare the tree $\mathcal{T}(f)$ with the tree $\mathcal{T}'(f)$ in these new coordinates.
The Newton maps corresponding to the vertices $v_1, \cdots,v_{i-1}$ are the same. 

At $v_i$, we have to consider the Newton maps
$$
\bdx=\bdy, \quad z=\bdy^{\mathbf{q}^i}(z_1-\mu_j).
$$
The Newton maps to be considered at $v_i'$ are given by the equations
$$
\bdx=\bdy, \quad z'=\bdy^{\mathbf{q}^i}(z'_1-(\mu_j-a)).
$$
Then $z_1=z_1'$ and along the corresponding edge nothing is changed.
Now at $v_i'$ we have also to consider the Newton map 
$$
\bdx=\bdy, \quad z'=\bdy^{\mathbf{q}^i}(z'_1-a)
\Longrightarrow
\bdx=\bdy, \quad z=\bdy^{\mathbf{q}^i}z_1'.
$$
In the change of variables, the monomial $\bdx ^{\balpha}z^{\beta}$ becomes $\bdy ^{\balpha+\beta\mathbf{q}^i}z_1'^{\beta}$.

We consider the transformation in the affine space 
$$
(\balpha,\beta)\mapsto (\balpha+\beta\mathbf{q}^i,\beta).
$$
A hyperplane with equations $\{P\alpha_j+Q_j\beta=0\}$ maps to $\{P\alpha_j'+(Q_j-P q_j^i)\beta=0\}$. Then the faces 
$\Gamma_{i+1},\cdots,\Gamma_n$, map to faces of the Newton diagram in the coordinates $(\bdy,z_1')$ 
and since  $\Gamma_{i+1},\cdots,\Gamma_n$
was a monotone polygonal path, it transforms to a monotone polygonal path. 

Now consider a face polynomial $\bdx^{k'}z^k \prod_l (z^P-\mu_l\bdx ^{\mathbf{Q}})^{m_l}$. In the change of coordinates it becomes 
$$
\bdy^{k'+\mathbf{q}^ik"}z_1'^{k}\prod_l (z^P-\mu_l\bdy^{\mathbf{Q}-\mathbf{q}^iP})^{m_l}.
$$ 
We consider the following Newton maps. 
Let $(\mathbf{u},u)$ be integers such that $1+\mathbf{u}\cdot\mathbf{Q}=uP$. 
We have $1+\mathbf{u}\cdot(\mathbf{Q}-P \mathbf{q}^i)=(u-\mathbf{u}\cdot\mathbf{q}^i)P$.
The Newton maps are
$$
\bdx=\mu_l^{\mathbf{u}}\bdx'^P,z=\bdx'^{\mathbf{Q}}(z_2+\mu_j^{u})\quad\text{and}
\quad
\bdy=\mu_l^{\mathbf{u}}\bdy'^P,z'_1=\bdy'^{(\mathbf{Q}-P \mathbf{q}^i)}(z'_2+\mu_j^{(u-\mathbf{u}\cdot\mathbf{q}^i)}).
$$
Since $z=\bdx^{\mathbf{q}^i}z_1'$, we have 
$\bdx'=\bdy', z_2=\mu_l^{\mathbf{u}.\mathbf{q}^i}z_2'$. Then we have essentially the same coordinates.

In conclusion, to get the tree of $f$ in P-good coordinates from 
$\mathcal{T}(f)$, we have to cut the edge~$e_i$ under $v_i$. 
We get two trees, $\mathcal{T}_a(f)$, the part which contains $v_i$ and $\mathcal{T}_u(f)$ containing~$e_i$. 
Then we stick again $\mathcal{T}_u(f)$ to $\mathcal{T}_a(f)$, sticking the edge $e_i$ on $v_i$ as an horizontal edge.
 We add a new vertical edge under $v_i$ decorated with $1$ and ending by an arrowhead decorated with~$(0)$. 
Since $v_i$ has a valency greater or equal to $3$ on $\mathcal{T}(f)$, it has a valency greater or equal to~$4$ on $\mathcal{T}'(f)$.

To put $f$ in P-good coordinates we made the choice of $a$. 
It is easy to verify that actually the tree doesn't depend on the choice of $a$. 
Then in this case we have unicity of $\mathcal{T}'(f)$ from $\mathcal{T}(f)$.

\begin{subcaso}\label{subcaso2}
$i=n$. 
\end{subcaso}

The edge $e_n$ under the vertex $v_n$ ends with an arrow decorated with $(k)$. Three cases may arise:
\begin{itemize}
\item If $k\neq 0$ then we are in P-good coordinates.
\item  If $k=0$ and the valency of $v_n$ is strictly greater than~$3$ then we are in P-good coordinates.
\item  If $k=0$ and the valency of $v_n$ is equal $3$ then we are not in P-good coordinates.
\end{itemize}
We study this last case. 
The face polynomial at $v_n$ is 
$$\bdx ^{k'}(z-a\bdx ^{\mathbf{q}^n})^m.$$
To eliminate the face $\Gamma_n$, we have to perform the (unique) change of variables 
$$z=z_1'+a\bdx^{\mathbf{q}^n}.$$
We considered the Newton map:
$$\bdx=\bdy, \quad z=\bdy^{\mathbf{q}^n}(z_1+a)$$
and we have $z_1'=\bdx^{\mathbf{q}^n}z_1$. The computation are the same than in Sub-Case~\ref{subcaso1}.

To get $\mathcal{T}'(f)$ from $\mathcal{T}(f)$, we do the inverse operation than in Sub-Case~\ref{subcaso1}. 
We delete the edge $e_n$. 
Denote by $e$ the horizontal edge attached to $v_n$. We cut $\mathcal{T}(f)$ in two pieces, separating $v_n$ and $e$ 
and we stick it back so that $e$ becomes the vertical edge under $v_n$. 
The vertex $v_n$ has now valency $2$ and has to be eliminated.

Note that we are not necessarily in P-good coordinates yet. We illustrate this fact in several examples:
\begin{itemize}
\item In Figure~\ref{fig:QO082a} we are in suitable coordinates and in Sub-Case~\ref{subcaso2}.
\item In Figure~\ref{fig:QO082c} we have eliminated the face $\Gamma_n$ but we are not yet in P-good
coordinates. We are in Sub-Case~\ref{subcaso1}.
\item In Figure~\ref{fig:QO082b} we are in P-good coordinates.
\end{itemize}

We have proven that if $f$ is in suitable coordinates such that its Newton polyhedron is a monotone polygonal path, 
then there is essentially a unique way to find P-good coordinates.
\end{proof}

\begin{figure}[ht]
\centering
\subfigure[]{
\includegraphics[scale=1]{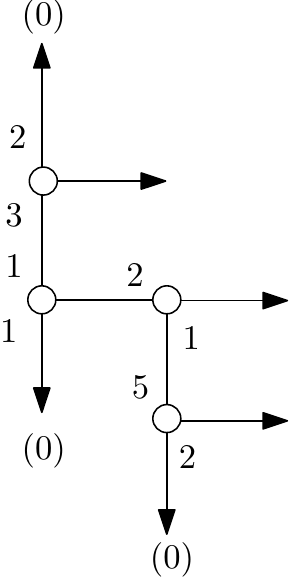}
\label{fig:QO082a}
}
\hfil
\subfigure[]{
\includegraphics[scale=1]{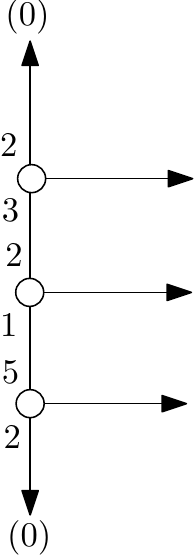}
\label{fig:QO082c}
}
\caption{Newton tree in  suitable coordinates and not P-good}
\label{fig1}
\end{figure}

\begin{figure}[ht]
\begin{center}
\includegraphics[scale=1]{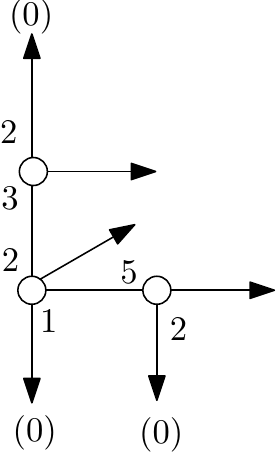}
\caption{Newton tree in P-good coordinates}
\label{fig:QO082b}
\end{center}
\end{figure}

\subsection{On $\nu$-Quasi-ordinary polynomials whose Newton tree 
ends only with arrow-heads.}
\mbox{}

Now we assume that $f$ is such that there exists a system of suitable coordinates such that $\mathcal{T} (f)$ has only arrowheads (no black boxes).

In that case, the Newton polyhedron of $f$ is a monotone polygonal path.
We can find a system of P-good coordinates for $f$. 
Let $\mathcal{T}_1 (f)$ denote the Newton tree in this system. 
It has also only arrowheads. Consider the Newton transforms of $f$ in this system of P-good coordinates;
their Newton polyhedra are monotone polygonal paths. 
Then we can find a system of P-good coordinates for them, and so on, see Proposition~\ref{prop-p-coord}. 
We can summarize it as follows.

\begin{definition}
A Newton tree is said to be a \emph{P-good tree} if
the vertices with decoration 
$p=1$ are at the bottom of the tree, connected to a vertical edge ending with an arrowhead decorated with $(0)$, 
and with valency strictly greater than $3$.
\end{definition}

\begin{proposition}
Starting from a Newton tree $\mathcal{T} (f)$ (for suitable system of coordinates), 
we can associate to $f$ a unique P-good tree.
\end{proposition}

\begin{qst}\label{qst-1}
Assume there are two suitable systems of coordinates such that the Newton trees
$\mathcal{T} (f)$ and $\mathcal{T}' (f)$ have only arrowheads. Do we get the same P-good tree?
\end{qst}

In order to answer this question, we distinguish three cases:

\begin{enumerate}

\item The change of variables doesn't change the Newton diagram of $f$. 
It doesn't change either the face polynomials, then $\mathcal{T} (f)=\mathcal{T}' (f)$. 
Then there is one P-good tree.

\item The change of variables modifies the Newton diagram as follows: 
The edges $\Gamma_1,\cdots,\Gamma_i$ remain unchanged and 
$\Gamma_{i+1},\cdots, \Gamma_n$ are replaced by a newface $\Gamma$. 
In any case $\mathcal{T}'(f)$ is not in P-good coordinates
and we are in Sub-Case~\ref{subcaso2} in the proof of Proposition~\ref{prop-p-coord}.
In order to get P-good coordinates we first have to come back to $\mathcal{T} (f)$. Then we have the same P-good tree.

\item The change of coordinates changes the Newton diagram as follows: It keeps unchanged $\Gamma_1,\cdots,\Gamma_i$, and replace $\Gamma_{i+1}$ by a face with same equations which hits the hyperplane $\{z=0\}$. Either one of 
$\mathcal{T} (f)$ or $\mathcal{T}' (f)$ is P-good and is the P-good tree of the other, or the two trees only differ by exchanging one vertical edge and one horizontal edge. They have the same P-good tree because in the P-good tree they are all horizontal.
\end{enumerate}

Hence, if the Newton tree $\mathcal{T} (f)$ for $f$, in suitable coordinates, has no black box
then we associate to it a unique P-good tree denoted by 
$\mathcal{T}_P (f)$ and the answer to Question~\ref{qst-1} is positive.

\begin{remark}
A P-good tree is not a minimal tree in the sense of Eisenbud and Neumann~\cite{en:85}. 
A minimal tree is unique if we forget about the direction of the edges, but not in the strong sense. A P-good tree is unique in the strong sense.
\end{remark}

This allows us to define the depth of $f$ in suitable coordinates such that $\mathcal{T} (f)$ has no black box.

\begin{definition}
The \emph{depth} of $f$ is the maximal length of horizontal paths in $\mathcal{T}_P (f)$,
denoted by~$\depth(f)$.
\end{definition}

\section {Quasi-ordinary polynomials}

\begin{notation}\label{inicial3}
 Let 
$g(\bdx,z) \in \bk[[\bdx]][z]$  be a 
polynomial of degree $m$ and regular of order $n\leq m$ 
with coefficients in the formal
power series ring $\bk[[\bdx]]$, that is 
$$
g(\bdx,z):=z^m+a_1(\bdx)z^{m-1}+\ldots+a_{m-n}(\bdx)z^n+a_{m-n+1}(\bdx)z^{n-1}+\ldots+ a_{m-1}(\bdx)z+a_m(\bdx),
$$
with $a_i(\textbf{0})=0$ for $m-n+1\leq i\leq m$ and $a_{m-n}(\textbf{0})\in \bk^*$ and
let $f(\bdx,z):=a_0 x_1^{n_1}\cdot\ldots\cdot x_d^{n_d} g(\bdx,z)$ where $ a_0\in \bk^*$. 
Applying  Weierstrass Preparation Theorem  to $g(\bdx,z)$ there exists a unit $u(\bdx,z) \in \bk[[\bdx]][z]$
and a Weierstrass polynomial 
$$
h(\bdx,z):=z^n+B_1(\bdx)z^{n-1}+\ldots+ B_{n-1}(\bdx)z+B_n(\bdx)
$$
such that $g(\bdx,z)=h(\bdx,z)u(\bdx,z)$. In fact 
$\deg_z u(\bdx,z)=m-n$,
e.g. see \cite[Chapter I,~p.~11]{GLS}. 
\end{notation}

\begin{definition}
A  polynomial $f\in \bk[[\bdx]][z]$  as in Notation \ref{inicial3} is called \emph{quasi-ordinary} if  its
discriminant $\Delta_z(f):=\res_z(f,\frac{\partial f}{\partial z})(\bdx)$ is
 $$\res_z(f,\frac{\partial f}{\partial z})(\bdx)=\bdx^{\balpha} \eta(\bdx)$$ with $\balpha\in \bn^d$ and $\eta(\textbf{0})\ne 0$.
\end{definition}

\begin{remark}
Since $\res_z(f,f_1f_2)=\res_z(f,f_1)\res_z(f,f_2)$, 
a  polynomial $f\in \bk[[\bdx]][z]$  as in Notation \ref{inicial3} is {quasi-ordinary} if and only if 
$g$ is quasi-ordinary.

If $f=f_1f_2\in \bk[[\bdx]][z]$ then $f_i$ is also quasi-ordinary because $\bk[[\bdx]]$ is a factorial ring and
the well-known  property of discriminants $\Delta_z(f_1f_2)=\Delta_z(f_1)\Delta_z(f_2)(\res_z(f_1,f_2))^2$.
\end{remark}

In this section we prove the following theorem.

\begin{theorem}\label{main2}
Let  $f\in \bk[[\bdx]][z]$ be a polynomial   as in Notation {\rm~\ref{inicial3}}. 
Then 
$f$ is quasi-ordinary if and only if there exists a suitable system of coordinates of $f$ such that
its Newton tree $\mathcal{T}(f)$ has only arrow-heads decorated with $(0)$ and $(1)$ and has no black boxes.

\end{theorem}

Quasi-ordinary power series were introduced by Zariski
using the discriminant. In fact quasi-ordinary power 
series are the natural generalization, in the sense of the 
factorization theorem given by Jung and Abhyankar~\cite{ab:55}, of algebraic curves.

 The fact that at each stage of the Newton process, using eventually automorphisms,
its Newton polyhedron is a monotone polygonal path is very useful.
It is one of the main ingredients in the proof in~\cite{aclm:05} of the monodromy
conjecture for hypersurfaces defined by quasi-ordinary power series in arbitrary dimension.

First we prove that if $f\in \bk[[\bdx]][z]$ is quasi-ordinary, 
the Newton process ends with monomials multiplied by a unit.

\begin{proposition}[\cite{lu:83}]\label{prop-lu}
Let $f\in \bk[[\bdx]][z]$ be a quasi-ordinary polynomial, then there exists a power series $b(\bdx)\in \bk[[\bdx]]$ such that 
$f(\bdx,z-b(\bdx))\in \bk[[\bdx]][z]$ is a $\nu$-quasi-ordinary polynomial.
\end{proposition}

 \begin{lemma}[{\cite[Lemma 3.16]{GMN:03}}]
If $f$ is a quasi-ordinary polynomial in $\bk[[\bdx]][z]$,
then there exists a system of coordinates 
such that its Newton polyhedron $\mathcal{N}(f)$
is a monotone polygonal path.
\end{lemma}

\begin{proof}
Since $f(\bdx,z)=\bdx^\bdn h(\bdx,z)u(\bdx,z)$ the Newton diagram $\mathcal{N}_+(f)$ is the Minkowski sum ot each 
of the Newton diagram of its factors it is enough to prove for $h$ which is regular of order $n$ in $z$.

We will work by induction on the order of $h$ in $z$. It is true for $n=1$.
We assume that it is true for all regular polynomial of order strictly less than $n$.
Since $h$ is a regular quasi-ordinary polynomial, by Proposition~\ref{prop-lu}, there exists a change of coordinates 
such that $h$ is $\nu$-quasi-ordinary with a $\nu$-proper face $\Gamma_1$.  
By Factorization Theorem~\ref{teorema1} and Corollary \ref{cor-factor-gamma}
there exists an elementary polynomial  $f_1$ in $\bk[[\bdx]][z]$ with Newton diagram parallel to $\Gamma_1$ and 
a polynomial $q(\bdx,z)$  in $\bk[[\bdx]][z]$ such that $h=f_1q$. 
Since $h$ is quasi-ordinary and $q$ is one of its factors then $q$ is a quasi-ordinary polynomial of order strictly less than $n$.
From the hypothesis, there exists a system of coordinates such that $\mathcal{N} (q)$ is a monotone
polygonal path, and this change of coordinates does not change  the Newton diagram of $f_1$. Using
the fact that the Newton diagram of a product is the Minkowski sum of the Newton diagrams of each factors,
we deduce that the Newton polyhedron of $h$ is a monotone polygonal path.
\end{proof}

In particular there exist P-good coordinates for quasi-ordinary  polynomials $f$.

 \begin{lemma}\cite[Chapter~3, Lemma 3.21]{aclm:05}
If $f$ is quasi-ordinary, after a Newton map $\sigma _{\Gamma, j}$,  
the total transform  $f_{\Gamma_1, j}:=\sigma _{\Gamma_1, j}(f)\in \bk[[\bdy]][z_2]$
 is a quasi-ordinary polynomial in $\bk[[\bdy]][z_2]$.
\end{lemma}

Furthermore the  \textbf{depth} of a quasi-ordinary  polynomial $f$ in P-good coordinates 
was defined similarly as in \cite[Chapter~3, Definition~24]{aclm:05}, and it decreases after Newton maps.

\begin{lemma}\label{lema-ends}
Let  $f\in \bk[[\bdx]][z]$ be a polynomial   as in Notation{\rm~\ref{inicial3}}. 
If
$f$ is quasi-ordinary then there exists a system of suitable coordinates such that
its Newton tree $\mathcal{T}(f)$ has only arrow-heads decorated with $(0)$ and $(1)$ and has no black boxes.
\end{lemma}

\begin{proof}
In fact in \cite[Chapter~3, Section~3.26]{aclm:05} it was defined Newton tree $\mathcal{T}(f)$ for 
of a quasi-ordinary  polynomials $f$ in P-good coordinates and in this tree all  ends 
are arrow-heads and if an arrow-head is not a dead-end then it is decorated with~$(1)$.
\end{proof}

\mbox{}

To prove the converse 
it is enough to prove that 
$f\in \bk[[\bdx]][z]$ and its \emph{polar} $f_z:=\frac{\partial f}{\partial z}\in \bk[[\bdx]][z]$ 
are  separated polynomials because after Theorem \ref{thm-sepcom} they are comparable series.

We assume that there is a system of suitable coordinates such that 
$\mathcal{T} (f)$ has no black box and arrowheads decorated with $(0)$ and $(1)$. We consider coordinates such that the Newton tree is 
$\mathcal{T}_P (f)$. Since P-good coordinates are suitable coordinates, these coordinates are suitable for $ff_z$.

Given a linear form $D$ given by $\sum_{j=1}^d a_j \alpha_j+ b\beta$, $a_j,b\in\bn$,
we define $f_D(\bdx,z)$ as the sum of the monomials $c_{\balpha,\beta}\bdx^{\alpha}z^{\beta}$,
$c_{\balpha,\beta}\neq 0$, for which $D(\balpha,\beta)$ is minimal. Note that $f_D$ is not a monomial
if and only there is a face of the Newton polyhedron contained in an affine hyperplane defined by~$D$.

\begin{lemma}
Assume that $D$ is a linear form such that 
$$
f_D(\bdx,z)=\bdx^{\mathbf{A}}z^B,
$$
with $B\neq 0$, then there is no face of $\mathcal{N}(f_z)$ contained in an affine hyperplane defined by~$D$.
\end{lemma}

\begin{proof}
Denote by $D_+$ the space limited by the first quadrant and the affine hyperplane defined by~$D$ 
which hits the Newton polyhedron of $f$ and doesn't contain the origin. 
We have
$$
f(\bdx,z)=f_D(\bdx,z)+\sum_{(\balpha,\beta)\in D_+,\beta>0}
c_{\balpha,\beta}\bdx^{\balpha}z^{\beta}+\sum_{(\balpha,\beta)\in D_+,\beta=0}
c_{\balpha,\beta}\bdx^{\balpha}z^{\beta},
$$
$$f_z(\bdx,z)=B\bdx^{\mathbf{A}}z^{B-1}+\sum_{(\balpha,\beta)\in D_+,\beta>0}
\beta c_{\balpha,\beta}\bdx^{\balpha}z^{\beta-1}.
$$
Then, there is no face of the Newton diagram of $f_z$ in an affine hyperplane defined by~$D$.
\end{proof}

\begin{lemma}
 Assume  $\Gamma_\ell$ is a face of the Newton diagram of $f$ 
which doesn't hit the hyperplane $\{z=0\}$ with face polynomial,
see{\rm~\eqref{eq-fgamma-ell}},
$$
f_{\Gamma_\ell }=
a_{\Gamma_\ell }\bdx^{\mathbf{n^\ell}}z^{n^\ell_{d+1}}\prod _{j=1}^{k^\ell}(z^{p_\ell}-\mu _j^\ell 
\bdx^{\mathbf{q^\ell}})^{m_j^\ell},
$$
the factors being irreducible in $\bk[[\bdx]][z]$, 
i.e. $\gcd (
\mathbf{q^\ell},p_\ell)=1$, $\mu_j^\ell\in \bk^*$ with $\mu_j^\ell\ne \mu_i^\ell$ and $a_{\Gamma_\ell }\in \bk^*$.

Then the Newton diagram of $f_z$ has a face $\Gamma'_\ell$  parallel to $\Gamma_\ell$ and
the corresponding initial form $(f_{z})_{\Gamma'_\ell }$ is
$$
a_{\Gamma_\ell }\bdx^{\mathbf{n^\ell}}\!\!   z^{n^\ell_{d+1}-1}
\!\!\!\left(\prod _{j=1}^{k^\ell}(z^{p_\ell}-
\mu _j^\ell 
\bdx^{\mathbf{q^\ell}})^{m_j^\ell-1}\right)\!\!\!
\left(\sum _{i=1}^{k^\ell}\left((n^\ell_{d+1}+p_\ell m_i^\ell)z^{p_\ell}
-n^\ell_{d+1}\mu_i^\ell \bdx^{\mathbf{q^\ell}}\right)\prod _{j\neq i}
(z^{p_\ell}-\mu _j^\ell \bdx^{\mathbf{q^\ell}})\right)\!\!.
$$
\end{lemma}

\begin{proof}
Same computation as before.
\end{proof}

\begin{lemma}
Assume $\Gamma_\ell$ is the face of the Newton diagram of $f$ which hits the hyperplane $\{z=0\}$ with face polynomial
$$
f_{\Gamma_\ell}=a_{\Gamma_\ell}\bdx ^{\bdn_\ell} \prod_{j=1}^{k^\ell} (z^{p_\ell}-\mu_j^\ell\bdx ^{\mathbf{q}_\ell})^{m_j^\ell}.
$$
Then if $k^l>1$ or if $k^l=1$ and $m^l_1>1$ the Newton diagram of $f_z$ has a face parallel to $\Gamma_\ell$, 
denoted by $\Gamma'_\ell$ and
\begin{equation}\label{eq-newton-fz}
(f_z)_{\Gamma'_\ell}= p_\ell a_{\Gamma_\ell}\bdx^{\bdn_\ell}z^{p_\ell-1}
\prod _{j=1}^{k^\ell} (z^{p_\ell}-\mu_j^\ell\bdx ^{\mathbf{q}_\ell})^{m_j^\ell-1}
\left(\sum_{i=1}^{k^\ell}m_i^\ell\prod_{j\neq i}(z^{p_\ell}-\mu_j^\ell\bdx^{\mathbf{q}}_\ell)\right). 
\end{equation}
If $k^\ell=1$ and $m_1^\ell=1$ then the Newton diagram of $f_z$ has no face parallel to $\Gamma_\ell$.
\end{lemma}

\begin{remark}
If $k^\ell=1$ and $m_1^\ell=1$ then $p_\ell>1$ since we are in P-good coordinates for $f$. 
Then either $f_z$ is divisible by $z^{p_\ell-1}$ or $\mathcal{N}(f_z)$ has faces which are not faces of $f$.
The proof of the above Lemma is straightforward.
\end{remark}

\begin{cor}
Consider the bicoloured Newton tree $\mathcal{T}(ff_z)$ coloured blue for $f$ and red for~$f_z$. 
Denote by $v_1,\cdots,v_n$ the vertices on the first vertical line. 

There exists $n_0$, $1\leq n_0\leq n$ such that for $1\leq j\leq n_0$ 
there is a blue or blue-red horizontal edge attached to $v_j$ and for $n_0<j\leq n$ 
there is no blue, neither blue-red horizontal edge attached to $v_j$. 

Attached to $v_j$, $1\leq j\leq n_0$, there are eventually blue horizontal edges ending with an arrow and/or blue-red horizontal edges ending with a vertex and always red horizontal edges.
\end{cor}

The  horizontal edges corresponding to
 the roots coming from the factor 
$$\sum _{i=1}^{k^\ell}\left((n^\ell_{d+1}+p_\ell m_i^\ell)z^{p_\ell}
-n^\ell_{d+1}\mu_i^\ell \bdx^{\mathbf{q^\ell}}\right)\prod _{j\neq i}
(z^{p_\ell}-\mu _j^\ell \bdx^{\mathbf{q^\ell}}),
 $$ of $(f_{z})_{ \Gamma'_\ell }$
 are red-coloured. The degree in $z$ is $k^\ell p_{\ell}$. 

\bigskip

\begin{proof}[Proof of $\Leftarrow)$ of Theorem{\rm~\ref{main2}}]
We use induction on $\delta:=\depth(f)$.
Let us start with the case $\delta=0$. 
We have $f=\bdx^{\bdn}z$ up to a unit and $f_z=\bdx^{\bdn}$. Then $f$ and $f_z$ are separated. 

For $\delta=1$, 
consider a vertex $v_j$ of the tree of $ff_z$. 
If $1<j\leq n_0$, there are exactly $k_j$ blue horizontal edges attached to $v_j$ ending by an arrow, and some red edges, but no blue-red edges. For  $n_0<j\leq n$, there are only red edges. Then $f$ and $f_z$ are separated.

Assume that for any $g$ in P-good coordinates with $\mathcal{T}_P(g)$ 
with no black box and arrowheads decorated with $(0)$ and $(1)$ of depth $\delta'<\delta$, $g$ and $g_z$ are separated.

Consider $f$ in P-good coordinates of depth $\delta$ and assume that $\mathcal{T}(f)$
 has no black box and has arrowheads decorated with $(0)$ and $(1)$. 
If $\delta>1$ there is a vertex $v_j$, $1\leq j\leq n_0$ such that there is a blue-red horizontal edge attached to $v_j$. 
We consider the corresponding Newton map $\sigma$.
Consider the polynomials $(f_z)_{\sigma}(\bdy,z_1)$ and $f_{\sigma}(\bdy,z_1)$, after the
change of variables to have $f_{\sigma}(\bdy,z_1)$ in P-good coordinates. 
Using see Lemma \ref{total-transf} we have
$(f_{\sigma})_{z_1}=\bdx^{\mathbf{q}}(f_z)_{\sigma}$. 

Therefore, $\mathcal{T}(f_{\sigma}\cdot(f_z)_{\sigma})$ and 
$\mathcal{T}(f_{\sigma}\cdot(f_{\sigma})_{z_1})$ are the same except for the decoration of the top arrow. 
Since $\depth(f_{\sigma})<\delta$, then $f_{\sigma}$ and $(f_{\sigma})_{z_1}$ are separated on 
$\mathcal{T}(f_{\sigma}\cdot(f_{\sigma})_{z_1})$. The tree $\mathcal{T}(f\cdot f_z)$ is obtained by
gluing  $\mathcal{T}(f_{\sigma}\cdot(f_z)_{\sigma})$ on $v_j$. 
We can apply that to every vertex where we have a blue-red edge attached. Then $f$ and $f_z$ are separated.
\end{proof}

\bigskip

This description allows us to study how  $f$ and its derivative $f_z$
separate on the Newton tree of $f$.

\begin{definition}
The edges at the bottom 
of the polyhedron ending with an arrow-head of multiplicity~$(0)$ 
(dead ends) are called {\it leaves}, and the vertices at the other end of the arrow-head \emph{leaf vertices}.
\end{definition}

\begin{theorem}\label{thm-ends}
Let  $f\in \bk[[\bdx]][z]$ be a polynomial   as in Notation{\rm~\ref{inicial3}} in a P-good system of coordinates  and
we consider the Newton tree
$\mathcal{T}(f)$ whose first vertical line 
has  $v_1,\ldots,v_s$ as  vertices.

\begin{enumerate}
\enet{\rm(D\arabic{enumi})}
\item\label{D1} The polar $f_z$ doesn't separate on the edges of the Newton tree of 
$f$ which are not leaves of $f$.
\item\label{D2} On each vertex $v_\ell$ of the Newton tree of $f$, different from a leaf vertex, the polar 
$f_z$ separates
with  order equal to $k^\ell p_{\ell}$, see~\eqref{eq-newton-fz}; recall $k^\ell +2$ 
is the valency of the vertex and $p_\ell$ the decoration under the vertex.
\item\label{D3} The polar  $f_z$ 
separates on a leaf vertex or on a leaf with total order
$\frac{k^s-1}{p_s}+p_s-1$
\end{enumerate}
\end{theorem}

\begin{proof}
It follows from the previous discussion.
\end{proof}

\begin{remark}
We don't know
anything on $f_z$ after its separation with $f$. In particular in general the derivative is
not quasi-ordinary. In the case $d=1$ this gives L\^e-Michel-Weber Theorem~\cite{lmw:89}.
\end{remark}

We can compute the discriminant of $f$ from its Newton tree. The following
result follows from a recursive application of \cite[Lemmas~5.7,~5.10]{aclm:09}, Proposition~\ref{prop-sep} and Theorem~\ref{thm-ends}.
 
\begin{proposition}\label{discrimant-tree}
Let $f$ be a quasi-ordinary Weierstrass polynomial of degree~$n$. 
Its discriminant can be computed from the polyhedron of $f$ by
$$
\Delta_z(f):=\res_z\left(f,f_z\right)(\bdx)=\bdx^{\mathbf{D}} u(\bdx), \textit{ with }   u(0)\ne 0,
$$
where 
$$
\mathbf{D}:=(D_1,\ldots,D_d)=\sum _v (\delta _v -2)\mathbf{N}'_v  -\sum _{v\text{\rm\ leaf vertex}}\frac{\mathbf{N}'_{v}}{p_{v}},
$$
where
$\delta_v$ is the valency of the vertex~$v$, 
and $\mathbf{N}'_v=\mathbf{\rho}*\mathbf{N}_v$ where $\mathbf{\rho}:=\prod_{w\text{\rm\ before }v}\mathbf{c}_w$ 
(recall that $\mathbf{c}_w$ is the $\gcd$'s of the vertex $w$).
\end{proposition}
 
This can be seen as a generalisation of Kouchnirenko theorem~\cite{kus,cn:96} and
as a generalization of \cite[Corollaire 5]{pgp:00} where  P.~Gonz{\'a}lez-P{\'e}rez 
gave information about the Newton polyhedron of the
resultant of two quasi-ordinary hypersurfaces satisfying an
appropriate non-degeneracy  condition.

\section{Transversal sections}\label{sec-trans-sect}

Consider a quasi-ordinary polynomial $f\in \bk[[\bdx]][z]$. Fix $i\in \{1,\cdots , d\}$. We follow this convention: if $\mathbf{p}$ is
a $d$-tuple, then $\mathbf{p}^i$ is the $(d-1)$-tuple $(p_1,\dots,p_{i-1},p_{i+1},\dots,p_d)$.
Let $K_i$ be an algebraic closure of $K((x_i))$. 
The $i$-\emph{transversal section} of $h$ is the quasi-ordinary polynomial
$h^i\in K_i[[\bdx^i]] [z]$ obtained from $h$ considering $x_i$ as a generic constant.
For $I\subset \{1,\cdots,d\}$ we denote $I':=\{1,\cdots,d\}\setminus I$; 
we can define recursively the $I$-\emph{transversal section} $h^I$ and if 
$I'$ has one element, say $j$, we call it \emph{$j$-curve transversal section} $h^{I_j}$.

Consider $f$ a quasi-ordinary polynomial in suitable coordinates such that its Newton tree has no black box. We first study the Newton trees of the transversal sections in the same system of coordinates. 

Then we assume the Newton tree of $f$ has only one arrow-head with positive multiplicity.
We show that we can retrieve its 
decorated Newton tree from  the Newton tree of its curve 
transversal sections in the same system of coordinates. 

We give examples where this is no 
more true when the Newton tree of $f$ has more than one arrow-head with positive multiplicity. Nevertheless  we can prove
 that we can retrieve the global decorations of the vertices of~$f$ from the decorations 
of the curve transversal sections. This result was a crucial ingredient
in the proof of the monodromy conjecture for quasi-ordinary singularities in arbitrary dimension~\cite{aclm:05}.

Let $f$ be a quasi-ordinary polynomial in suitable coordinates such that the Newton tree $\mathcal{T}(f)$ 
has no black box. We are going to describe how we can find the tree of the $i$-transversal section $f^i$ 
in the same coordinates.

\subsection{Construction of $\mathcal{T}_{\mathcal{N}}(f^i)$.}
\mbox{}

The general procedure is that we can copy  $\mathcal{T}_{\mathcal{N}}(f)$ 
and erase the $i^{\text{th}}$-component of the local numerical data except in $3$ cases
which we develop below.

\begin{caso}\label{caso1}

Along  $\mathcal{T}_{\mathcal{N}}(f)$, there is a sequence of consecutive vertices where 
the local numerical data coincide after erasing the $i^{\text{th}}$-coordinate (and dividing
by the gcd).
This means that the two faces corresponding to these vertices project on the same face. 
Then the chain of vertices \emph{project}
on  $\mathcal{T}_{\mathcal{N}}(f^i)$ on the same vertex decorated with the common normalized numerical data.
\end{caso}

\begin{caso}\label{caso2}  The local numerical data at $v$ on  $\mathcal{T}_{\mathcal{N}}(f)$  satisfy
$\gcd (\mathbf{q}^i,p)=:d^i$, 
with $d^i>1$, and let $q'_j:=\frac{q_j}{d^i},  j \in I'$, and $p':=\frac{p}{d^i}$. 

Assume that the face polynomial at $v$ is 
$$
\prod _j (z^p -\mu _j\mathbf{x}^{\mathbf{q}})^{m_j}.
$$
On  $\mathcal{T}_{\mathcal{N}}(f^i)$, there is a vertex $v^i$ decorated with $(\mathbf{q'},p')$ and
with face polynomial
$$
\prod _j \prod _{\{\zeta \vert \zeta^{d^i}=1\}}\left(z^{p'} -\zeta \mu _jx_i^{\frac{q_i}{d_i}}(\mathbf{x}^i)^{\mathbf{q'}}\right)^{m_j}.
$$
Then if $k$ horizontal edges arise from $v$, $kd^i$ arise from $v^i$.
\end{caso}

\begin{remark}
Case~\ref{caso1} and~\ref{caso2} are not exclusive. 
The number of horizontal edges at $v^i$ is obtained as the sum, along
the vertices $v$ projecting on $v^i$,
of $d^i(v)$ times the number of horizontal edges arising from $v$.
\end{remark}

\begin{example}\label{ex-sec-trans-1}
We have in Figure~\ref{fig:desQO8nn1} the tree of
$$
f(x_1,x_2,z):=(z^2-x_1^3x_2)(z^2-x_1^3x_2^4)(z^2-x_1^5x_2^6)\in K[[x_1,x_2]][z];
$$
Figures~\ref{fig:desQO8nn2} and~\ref{fig:desQO8nn3} show the trees
for $f(x_1,x_2,z)\in K_1[[x_2]][z]$ and $f(x_1,x_2,z)\in K_2[[x_1]][z]$, respectively.
This example illustrates Cases~\ref{caso1} and~\ref{caso2}. 
More explicitely, the tree for $K_2[[x_1]][z]$ illustrates Case~\ref{caso1} 
and the tree for $K_1[[x_2]][z]$ illustrates Case~\ref{caso2}.
\begin{figure}[ht]
\centering
\subfigure[{$K[[x_1,x_2]][z]$}]{
\includegraphics[scale=1]{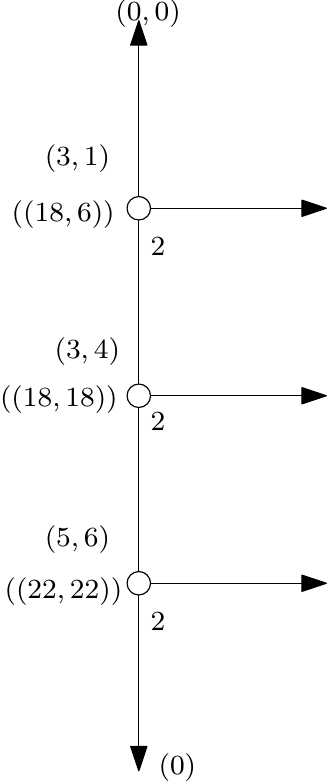}
\label{fig:desQO8nn1}
}
\hfil
\subfigure[{$K_1[[x_2]][z]$}]{
\includegraphics[scale=1]{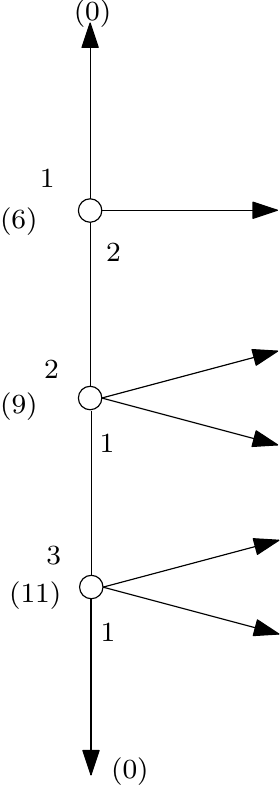}
\label{fig:desQO8nn2}
}
\hfil
\subfigure[{$K_2[[x_1]][z]$}]{
\includegraphics[scale=1]{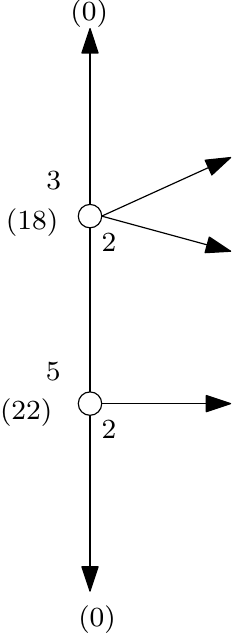}
\label{fig:desQO8nn3}
}
\caption{Newton trees of Example~\ref{ex-sec-trans-1}}
\label{desQO8nn}
\end{figure}
\end{example}

\begin{caso}\label{caso3}
We have $\mathbf{q}_\ell^i=\mathbf{0}^i$ for the first $h$ vertices of 
$\mathcal{T}_{\mathcal{N}}(f)$. 
That means that the projection of the face $\Gamma_\ell$ of $\mathcal{N}(f)$  has equation $p x_j=0$, 
$j \in I'$;
the polynomial $f_{\Gamma_\ell} (\bdx,z )$ doesn't depend on $x_j$, $j\in I'$, that is
$$
f_{\Gamma_\ell} (\bdx^i,z )=z^{k_\ell}\bdx^{\mathbf{N}}\prod_{j=1}^{r_\ell}(z^{p_\ell}-\mu_j^\ell x_i^{q_\ell})^{m_j^\ell},
$$
i.e., $f^i$ decomposes in $K_i[[\bdx^i]][z]$ in $m:=\sum_{\ell=1}^h p_\ell\sum_{j=1}^{r_\ell} m_j^\ell$
factors based at the $p_\ell$-roots $\mu_j^\ell x_i^{q_\ell}$ of $z$-degree $m_j^\ell$, and one
extra factor based at $z=0$ of degree $k_h$ (if $k_h>0$).
We proceed as follows. We keep $m$
void trees and if $k_h>0$ 
we consider also a tree for $z=0$, where we erase 
the vertices $v_1,\dots,v_h$ and we keep the upper arrow.
\end{caso}

\subsection{Construction of $\mathcal{T}(f^i)$ without decorations}
\mbox{}

Once we have obtained the vertical Newton trees we explain how to obtain
the trees $\mathcal{T}^i(f)$, when we see $f\in K_i[[\bdx^i]][z]$. 
We proceed by induction of $\depth(f)$. If the depth is one we have
$\mathcal{T}^i(f)=\mathcal{T}_{\mathcal{N}}^i(f)$ with the following exception.
If we are in Case~\ref{caso3}, we replace the empty trees by the ones
in~Figure~\ref{Q01}.

If the depth is greater than one, we start with $\mathcal{T}_{\mathcal{N}}(f^i)$.
Each arrow of this tree is related to one Newton transformation 
$\sigma:=\sigma_{\Gamma,\mu}$ and by the induction hypothesis we may
assume that $\mathcal{T}(f_\sigma^i)$ are constructed $\forall\Gamma,\mu$. 
In order to recover $\mathcal{T}(f^i)$ we proceed
as in Step~\ref{paso3} of~\S\ref{constr-trees} with the following caveats.

If we are not in Case~\ref{caso3} and $\mathcal{T}(f_\sigma^i)$ is a disjoint union
of $k_{\Gamma,\mu}$ trees, then we perform Step~\ref{paso3} for each connected component.
Hence we will obtain $\prod_{\Gamma,\mu} k_{\Gamma,\mu}$ disconnected trees.

If we are in Case~\ref{caso3}, we proceed in the same way for the tree corresponding
to~$z=0$ (if $k_h>0$). Each void tree is associated to a pair $(\Gamma_\ell,\mu_j^\ell)$,
$\ell\leq h$, and we simply take $p_\ell$ copies of the tree $\mathcal{T}(f_\sigma^i)$,
$\sigma=\sigma_{\Gamma_\ell,\mu_j^\ell}$.

\begin{remark}
In order to construct Newton maps, we followed some conventions in \S\ref{sec-newton} which may not pass
to transversal sections. It is easily seen that the results do not depend
on the particular choice of conventions. In particular, everything works if
one does not choose a tuple $\mathbf{u}$  in \S\ref{sec-newton} but a $p^{\text{th}}$-root
of $\mu_j$ instead.
\end{remark}

\subsection{Decorations of $\mathcal{T}(f^i)$.}
\mbox{}

In order to finish the construction of the Newton trees $\mathcal{T}^i(f)$ we need
to compute the decorations. In order to do it in a simpler way we define
new decorations to the Newton tree.

\begin{definition}
The \emph{prime local decorations} $\langle R_1^v,\dots,R_d^v;p^v\rangle$
of a vertex~$v$ are defined as follows. They coincide with the local decoration
if $v$ is in the first vertical tree. If not, they are computed with the following
recursive formula (as in Lemma~\ref{lema-newton-tree-Q}) where $w$ is the preceding
vertex:
$$
R_j^v:=q_j^v+\frac{p^v R_j^w p^w}{\gcd(R_j^w,p^w)^2}.
$$
\end{definition}
Note that it is possible to compute local decorations from prime local decorations
(and viceversa). Moreover, they coincide if the decorations are \emph{coprime enough},
e.g., in the curve case. The reason to define these decorations is that they behave
better when passing to transversal section since they really depend on the
quotients $\frac{q_j}{p}$ and not on the pairs $(q_j,p)$.

In order to decorate the Newton trees of the transversal sections the strategy
is as follows:

\begin{enumerate}
\enet{\rm(TSNT\arabic{enumi})}
\item In  $\mathcal{T}(f)$ compute the prime local decorations from the local decorations.
\item Pick a vertex $v^i$ in $\mathcal{T}^i(f)$, and consider the prime local decorations
of the vertex $v$ in $\mathcal{T}(f)$ which originates~$v^i$. Forget the $i^{\text{th}}$
coordinate and make the decoration coprime as in Case~\ref{caso2}.
\item Obtain the local decoration of $v^i$ from the prime local decorations.
\end{enumerate}

\subsection{Examples of transversal sections.}
\mbox{}

We illustrate the above theory with some examples.

\begin{example}\label{ex-sec-trans-2}
The tree of Figure~\ref{fig:desQO9nn1} corresponds to
$$
f(x_1,x_2,x_3,z)=(z^7-x_1^2)^2(z^3-x_1^5x_2x_3)+x_1^{10}x_2x_3\in K[[x_1,x_2,x_3]][z]
$$
and illustrates Case~\ref{caso3}.
\begin{figure}[ht]
\centering
\subfigure[{$K[[x_1,x_2,x_3]][z]$}]{
\includegraphics[scale=1]{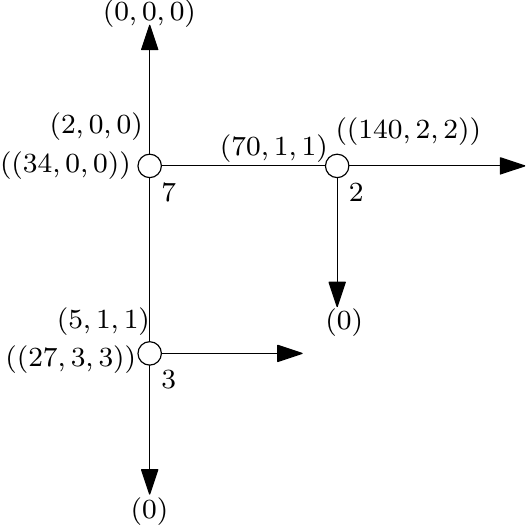}
\label{fig:desQO9nn1}
}
\hfil
\subfigure[{$f_0$}]{
\includegraphics[scale=1]{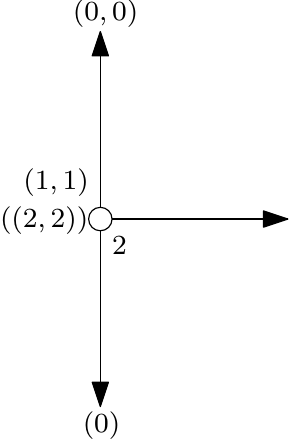}
\label{fig:desQO9nn2}
}
\hfil
\subfigure[{$f_j$}]{
\includegraphics[scale=1]{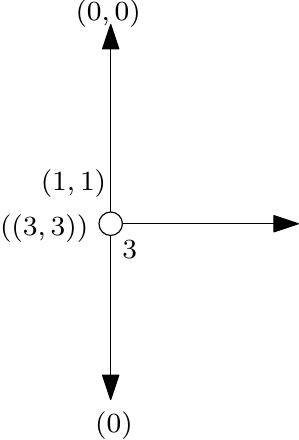}
\label{fig:desQO9nn3}
}
\caption{Newton trees of Example~\ref{ex-sec-trans-2}.}
\label{desQO9nn}
\end{figure}
If we consider $f\in K_1[[x_2,x_3]][z]$, we have $f(x_1,0,0,z)=(z^7-x_1^2)^2 z^3$, i.e,
$f=f_0\prod_{j=1}^7 f_j$ where $f_0(x_1,0,0,0)=0$ and $\{f_j(x_1,0,0,0)\}_{j=1}^7=\{b\in K_1\mid b^7=x_1^2\}$.
Hence we obtain $8$ disjoint Newton trees, one for $f_0$ and $7$ equal trees for $f_j$, $j=1,\dots,7$. 
\end{example}

These following examples illustrate Case~\ref{caso3} at the second step of the algorithm.

\begin{example}\label{ex-sec-trans-3}
Figure~\ref{fig:desQO10nn1} shows the tree of 
$$
f(x_1,x_2,z)=((z^7-x_1^2x_2^3)^2-x_1^5x_2^6)^2+x_1^{11}x_2^{13}\in K[[x_1,x_2]][z],
$$
while Figure~\ref{fig:desQO10nn2} shows the one for $f\in K_1[[x_2]][z]$.
\begin{figure}[ht]
\centering
\subfigure[{$K[[x_1,x_2]][z]$}]{
\includegraphics[scale=1]{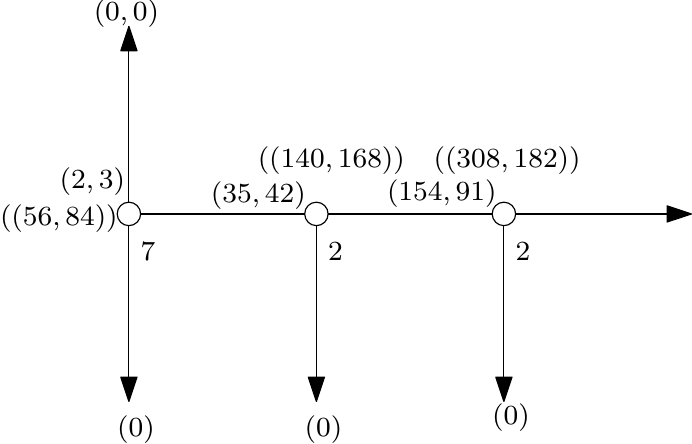}
\label{fig:desQO10nn1}
}
\hfil
\subfigure[{$K_1[[x_2]][z]$}]{
\includegraphics[scale=1]{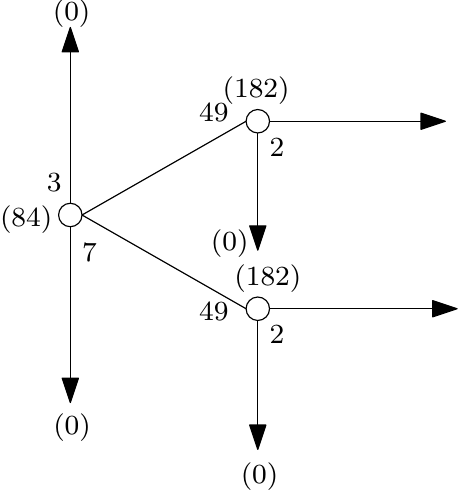}
\label{fig:desQO10nn2}
}
\caption{Newton trees of Example~\ref{ex-sec-trans-3}.}
\label{desQO10nn}
\end{figure}
\end{example} 

\begin{example}\label{ex-sec-trans-4}
Figure~\ref{fig:desQO11nn1} shows the tree of 
$$
f(x_1,x_2,z)=(z^2-x_1^2 x_2^3)^6+ (z^2-x_1^2 x_2^3)^3 x_1^7 x_2^9+x_1^{15} x_2^{19}\in K[[x_1,x_2]][z].
$$
while Figure~\ref{fig:desQO11nn2} shows the one for $f\in K_1[[x_2]][z]$.
\begin{figure}[ht]
\centering
\subfigure[{$K[[x_1,x_2]][z]$}]{
\includegraphics[scale=1]{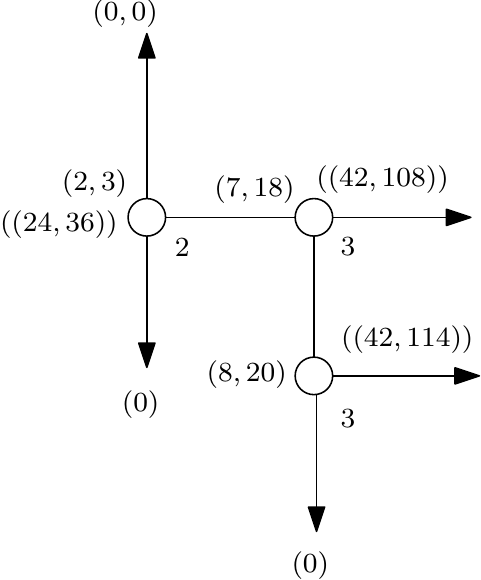}
\label{fig:desQO11nn1}
}
\hfil
\subfigure[{$K_1[[x_2]][z]$}]{
\includegraphics[scale=1]{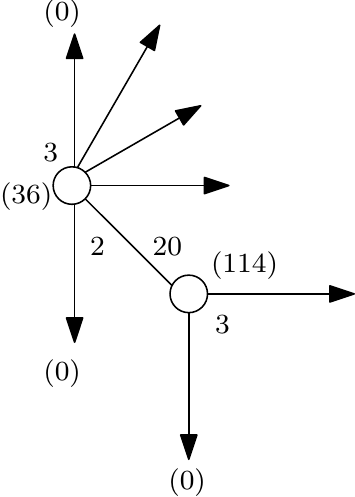}
\label{fig:desQO11nn2}
}
\caption{Newton trees of Example~\ref{ex-sec-trans-4}.}
\label{desQO11nn}
\end{figure}
\end{example}

\subsection{Curve transversal sections}
\mbox{}

We summarize the previous process showing the computations for curve transversal sections.
Now we are describing  the Newton trees of the curve transversal sections in the case where, in suitable coordinates,
the Newton tree of $f$ has only one arrow-head with positive multiplicity and no black box.
Consider the first vertex and fix $i\in\{1,\dots,d\}$. There are three cases. 

\begin{vertex1}
$q_i\neq 0 ,\ \gcd(q_i,p)=1$.
\end{vertex1}

In this case, we have a first vertex decorated with $(q_i,p)$
and with a unique horizontal edge starting from this first vertex.

\begin{vertex1}
$q_i\neq 0,\ \gcd(q_i,p)=:c_i> 1$.
\end{vertex1}

We have a first vertex decorated with 
$(q'_i,p_i)$, $q'_i:=\frac{q_i}{c_i}$, $p_i:=\frac{p_i}{c_i}$, and $c_i$ horizontal edges starting from the vertex. 

\begin{vertex1}\label{fv3}
$q_i=0$.
\end{vertex1}
We are not studying the transversal section at the origin for $z$, but at some other points. 
There are $p$ of them. In this case we have $p$ Newton trees which begin eventually with the next vertex
(as far as this situation is not repeated again).

Now we apply a Newton map $\sigma_{\Gamma,\mu}$ to go to the following vertex 
whose decorations in $\mathcal{T}_{\mathcal{N}}(f_{\Gamma,\mu})$ are
$((r_1,\dots,r_n),p^1)$. 
The local numerical data $(Q_1,\cdots,Q_d,p^1)$ 
and the prime local numerical data
$\langle R_1,\dots,R_d;p^1\rangle$ of the second vertex
satisfy 
$$
Q_i=p^1 p \frac{q_i}{c_i}+r_i\quad
R_i=p^1 p \frac{q_i}{c_i^2}+r_i.
$$
Again we have three cases.

\begin{vertex2}
$\gcd(Q_i,p^1)=1$, $\forall i$.
\end{vertex2}

For all the transversal sections with a first vertex decorated with $(q_i^t,p^t)$
we add a new vertex to each of the edges 
decorated with $(r_i+q_i^tp^tp^1,p^1)$ and we add a unique edge to this vertex. If we do not have a first vertex
(the case of First Vertex~\ref{fv3})
we begin the tree
with a first vertex decorated with $(r_i,p^1)$ 
and one horizontal edge starting from the vertex.

\begin{vertex2}
For some $i$, $\gcd(Q_i,p^1)=:c_i^1>1$ and $r_i\neq 0$  
\end{vertex2}
For the
corresponding transversal section, if  we have a first vertex decorated with 
$(q_i^t,p^t)$  we add a new vertex decorated $(r_i+p^t p^1 \frac{q_i^t}{c_i},\frac{p^1}{c_i^1})$ to all the edges and
$c_i^1$ horizontal edges starting from the vertex.  
If we do not have a first vertex we begin the tree
with a first vertex decorated with $(\frac{r_i}{c_i^1},\frac{p^1}{c_i^1})$ and we add $c_i^1$ horizontal edges. 

\begin{vertex2}
For some $i$ we have $r_i=0$. 
\end{vertex2}

If  we have a first vertex decorated with 
$(q_i^t,p^t)$, we stay on the same vertex. Then if we had already $p$ 
edges we should have $pp^1$ edges starting from the vertex.
If we don't have already any vertex then we will have $pp^1$ trees starting eventually
with the next vertex. 

The global process continue with all their corresponding  Newton maps. Note that in the curve
case prime local decorations and local decorations coincide.

\subsection{Obtaining $\mathcal{T}(f)$ from $\mathcal{T}(f^{I_i})$.}
\mbox{}

Now we will see how one can retrieve the quasi-ordinary 
singularity from the transversal sections.
This is not always possible as shown in the following examples.

\begin{example}\label{ex-duple}
The quasi-ordinary polynomial
\begin{align*}
f_1=&((z^3-x_1^2)^2+x_1^{25} x_2^{11})((z^3-x_1^4)^2+x_1^{25}x_2^5)\\
f_2=&((z^3-x_1^2)^2+x_1^{25}x_2^5)((z^3-x_1^4)^2+x_1^{25} x_2^{11}) 
\end{align*}
have the same transversal sections but they do not have the same decorated Newton tree.
which are displayed in Figure~\ref{dessinduple}.

\begin{figure}[ht]
\centering
\subfigure[{$f_1$}]{
\includegraphics[scale=1]{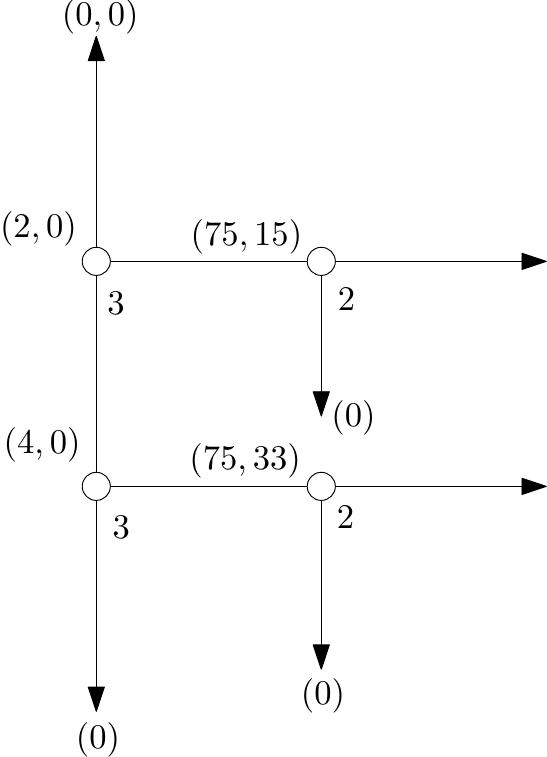}
\label{fig:dessinduple-1}
}
\hfil
\subfigure[{$f_2$}]{
\includegraphics[scale=1]{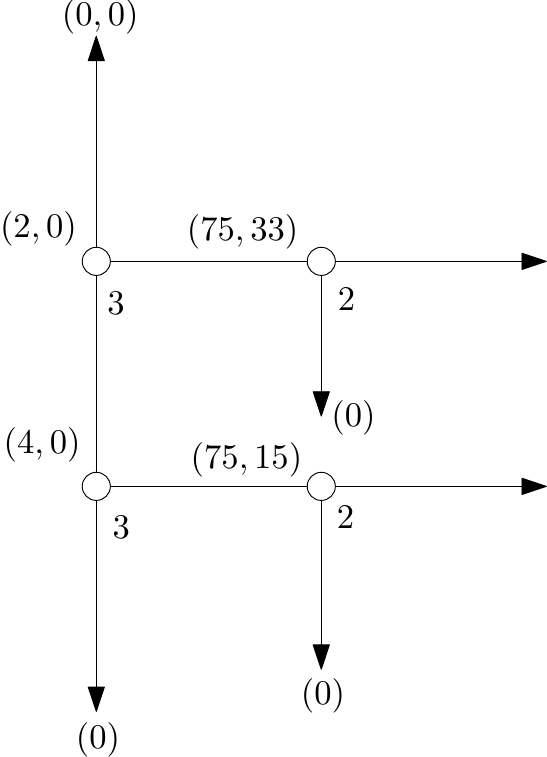}
\label{fig:dessinduple-2}
}
\caption{}
\label{dessinduple}
\end{figure}
\end{example}

\begin{example}
Let $f_n=z^n-x_1 x_2$. If $n_1\neq n_2$ then the decorated Newton tree of $f_{n_1}$ is not the same than the decorated
Newton tree of $f_{n_2}$. If we do not keep the system of coordinates,
i.e. exchanging the roles of $z$ and $x_i$ in each transversal section, we get
the same trees, see Figure~\ref{dessinduple2}, for any~$n$.

\begin{figure}[ht]
\begin{center}
\includegraphics[scale=1]{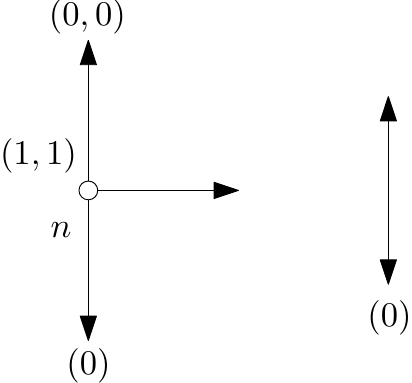}
\caption{}
\label{dessinduple2}
\end{center}
\end{figure}
\end{example}

\begin{example}\label{ex-duple1}
The transversal sections of the quasi-ordinary polynomial 
\begin{align*}
f_1=&((z^2-x_1^3 x_2)^2+x_1^5 x_2^3 z)((z^2-x_1^3 x_2^4)^2+x_1^6x_2^9 z)\\
f_2=&((z^2-x_1^3 x_2^4)^2+x_1^5 x_2^9z)((z^2-x_1^3 x_2)^2+ x_1^6 x_2^3 z) 
\end{align*}
have the same decorated Newton trees but the two germs
have different decorated Newton trees, see Figure~\ref{dessibduple1}.

\begin{figure}
\begin{center}
\includegraphics[scale=1]{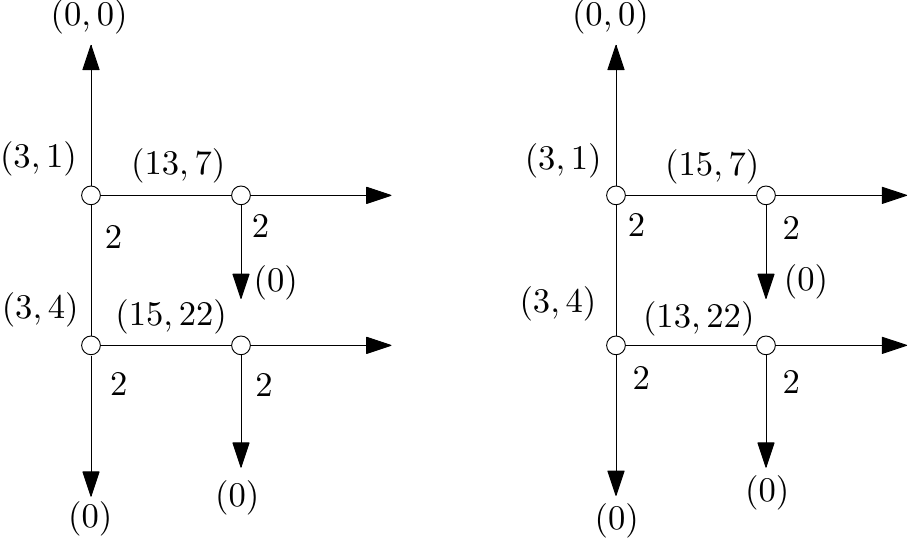}
\caption{}
\label{dessibduple1}
\end{center}
\end{figure}

\end{example}

\begin{definition}
The \emph{multiplicity} of a tree is the sum over all vertices~$v$ of the number
of horizontal edges arising from~$v$.
\end{definition}

\begin{theorem}\label{thm-recover}
Assume that $f$ is a quasi-ordinary polynomial in $\bk[[\bdx]][z]$ whose Newton tree in 
suitable coordinates has only one arrow-head with positive multiplicity and no black box.
 Then, there is a unique way to recover 
 the decorated Newton tree of $f$ from the 
 decorated Newton trees, in the same system of coordinates, of all its curve transversal sections.
 \end{theorem}
 
\begin{proof}
The proof can be done by induction on the sum of the multiplicities of the curve transversal sections.
If this sum is equal to one, it means that we may assume
$f_\Gamma(\bdx,z)=\bdx^{\mathbf{N}}(z^p-x_1^q)$ and the result follows immediately.

Let us assume that the result is true when the sum of the multiplicities is less than $n$, and assume
this sum for $f$ equals $n>1$. In this case
$$
f_\Gamma(\bdx,z)=\bdx^{\mathbf{N}}(z^p-\bdx^{\mathbf{q}})^{m},
$$
for $q_1\geq\dots q_k>q_{k+1}=\dots= q_d=0$; recall that $f$ hasn only
one arrow-head with positive multiplicity and no black box. Then we must have either $m\geq 2$ or $k\geq 2$.

If $k=1$ the results follows easily by induction. Let us consider the case $k\geq 2$.
s we have seen before,
after the Newton map we get in general 
bunches of disconnected trees.
Moreover, all the trees in the bunch corresponding to the $i^{\text{th}}$-coordinate
are equal, so we need only to retain one of them and the number of trees in the bunch.

We start the construction of the tree of $f$ from the trees of the curve transversal sections.
 We consider all curve transversal sections with a complete tree. There is at least one. We consider the first vertices of these trees. We consider the decorations
$(q_i,p_i)$ of these vertices. 
Let $p=\max \{p_1,\dots,p_d\}$. We know that $p_i$ divides $p$. 
Let $p=p_i c_i$. Then the first vertex of the tree of $f$ is decorated by $((q_1 c_1,\cdots,q_k c_k,0\cdots),p)$.

We consider the Newton tree
of the transforms of the curve transversal sections
of~$f$. The sum of the multiplicities is strictly less than~$n$ and hence
we can apply induction hypothesis and all the results in this sections.

Then we can reconstruct our decorated tree from the transversal sections.
There are many branches in the transversal sections, but because there are all the same
there is a unique way to reconstruct the tree of $f$. 
\end{proof}

Now we prove that we have all the information of the global numerical data can be found 
in the trees of the transversal sections.

\begin{proposition}
For all vertex $v$ of the tree of a quasi ordinary polynomial, and all $i=1,\cdots ,d$, 
the terms $\frac{N_{i,v}}{c_i}$  can be retrieved from the decorations 
of the trees of the transversal sections, except
the order $N$ in $z$ of $f$ when the transversal section is not divisible by the variables $x_j$.
\end{proposition}

\begin{proof}
The proof relies on the fact that the global decorations are given in terms of the equations
of the faces of the Newton diagrams.
If the vertex we are considering corresponds to a face of the  Newton diagram of $f$,  
we have the following cases:
\enumerate
\item
By projection  on $\{x_i=0\}$ the face is parallel to the $z$-axis: 
In this case, the vertex is the first vertex on
the top of the tree and its decoration is $((0,\cdots, N_i, \cdots ,0))$. The vertex will
disappear, but the decorations are $0$.

\item 
By projection the face $\Gamma$ gives a face of the Newton diagram of the transversal section. 
Either this face is different from the projections of faces intersecting $\Gamma$, and it will give
a vertex with the decoration $((N_1,\cdots ,N_d))$, 
or it is the same face as one of the projections, then the vertex will disappear, but the vertex
representing the face have the desired decorations.

Now, if we are on a Newton diagram which appeared in the Newton process. 
We have again different cases:
\begin{enumerate}
\item The Newton process is a linear change of variable for $z$. 
In this case there are two consecutive vertices 
which appear with the same $N_j$, $j\neq i$.
Only one appears in the transversal section, but its bears the decoration we are
interested in.
\item The Newton process is a Newton map for the transversal section, 
with the same roots. In this case, either we find the vertex in the transversal sections,
or it disappears because there are two faces which project on the same. In any case
the remaining vertex gives the decoration we need.
\item The Newton process is a Newton map for the transversal sections
with more roots. This case is the same as the case above but we have many copies
of the Newton diagram. Then the decoration appears many times.\qedhere
\end{enumerate}
\end{proof}

We can notice that not only we recover the decorations of the tree of $f$, but
no new decorations appear in the transversal sections.

\begin{remark}
A particular useful case of the preceeding proposition is that all the $N_i$'s can be retrieved
using the germs of curves obtained setting all variables $x_j$ except one as constant. This
is a fundamental fact in the proof of the monodromy
 conjecture for quasi-ordinary polynomial \cite{aclm:05}.
\end{remark}

 \bibliographystyle{amsplain}

\end{document}